\documentclass[12pt, reqno]{amsart}

\usepackage{amsmath, amssymb, amscd, amsthm, comment, amsxtra}
\usepackage{graphicx}
\usepackage{subfigure}

\usepackage[pdftex, linktocpage=true]{hyperref}
\usepackage{euscript}

\usepackage[format=plain,labelfont=bf,up,width=.99\textwidth]{caption}

\input{psfig}

\theoremstyle{plain}
\newtheorem{thm}{Theorem}[section]

\newtheorem{lem}[thm]{Lemma}
\newtheorem{prop}[thm]{Proposition}
\newtheorem{remark}[thm]{Remark}
\newtheorem*{conj}{Conjecture}
\newtheorem{cor}{Corollary}[thm]
\theoremstyle{definition}

\newtheorem*{mainthm}{Theorem A: Exponential splitting}
\newtheorem*{exthm}{Theorem B: Existence of separatrices}

\newcommand{\vareps}{\varepsilon}
\newcommand{\He}{H_\epsilon}
\newcommand{\Fe}{F_\epsilon}

\newcommand{\sech}{\operatorname{sech}}

\newcommand{\eps}{\epsilon}

\newcommand{\cD}{{\mathcal D}}
\newcommand{\cB}{{\mathcal B}}

\newcommand{\cX}{{\mathcal X}}

\newcommand{\CC}{{\Bbb C}}
\newcommand{\RR}{{\Bbb R}}

\newcommand{\ZZ}{{\Bbb Z}}
\newcommand{\NN}{{\Bbb N}}
\newcommand{\DD}{{\Bbb D}}

\title[Exponential splitting near period-doubling in area-preserving maps]{Exponentially small splitting of separatrices near a period-doubling bifurcation in area-preserving maps\\[8pt]
}

\author{Denis Gaidashev}
\address{Uppsala University, Uppsala, Sweden}
\email{gaidash@math.uu.se}

\author{Marina Gonchenko} 
\thanks{The second author  has been partially supported by Juan de la Cierva-Formaci\'on Fellowship $FJCI-2014-21229$.  M. Gonchenko warmly thanks the Department of Mathematics of Uppsala University for their hospitality and support; during her stay at Uppsala University, M. Gonchenko was also partially supported by  the Knut and Alice Wallenberg Foundation grant $2013-0315$.}
\address{Universitat de Barcelona, Barcelona, Spain}
\email{mgonchenko@gmail.com}

\subjclass[2010]{}
\keywords{}
\date{\today}

\begin{document}
\begin{abstract}
We consider the conservative H\'enon family at the period-doubling bifurcation of its fixed point and demonstrate that the separatrices of the fixed saddle point nearing the bifurcation split exponentially: given that $\lambda_+$ is the smaller of the eigenvalues of the saddle point, the angle between the separatrices along the homoclinic orbit satisfies
$$\sin \alpha =  O(e^{-{\pi^2 \over \log |\lambda_+|}})+ O\left( e^{-2 (1-\kappa)  {\pi^2 \over  \log |\lambda_+|}} \right),$$
for any positive $\kappa<1$.
\end{abstract}

\maketitle

\section{Introduction} \label{sec:intro}

Period-doubling bifurcation in area preserving maps of $\RR^2$ has been a focus of many works.  A period-doubling bifurcation in an area-preserving one-parameter family $H_t$ happens when the complex-conjugate eigenvalues of a linearization of the map $H_t^{2^k}$ at a $2^k$-periodic point pass to the real line through the value $-1$. The corresponding stable elliptic periodic point becomes a flip saddle: at the same time a period $2^{k+1}$ stable elliptic  periodic orbit is ``born'' (passes from $\CC^2$ to $\RR^2$).

It has been demonstrated numerically, that a typical family of area-preserving maps undergoes a period doubling cascade  which accumulates on a particular, ``universal'', or infinitely renormalizable, map $H_{t^*}$ in the family \cite{DP,Hel,BCGG,Bou}. To date there are no rigorous results about existence or genericity of such cascades. 

In the dynamic plane, the $2^k$ unstable periodic orbits, $k \in \NN$, of the map $H_{t^*}$ accumulate on a Cantor set (cf \cite{GJ1,GJ2}). Existence of such a map in every typical family in an appropriate functional space has been demonstrated in \cite{EKW1,EKW2} and \cite{GJ3}. The properties of the Cantor set have been studied in the renormalization framework in \cite{GJM}. In particular, it has been demonstrated that the Cantor sets are stable in the sense of vanishing Lyapunov exponents, and that they are  rigid: the dynamics on the Cantor sets $C_{H_{t^*}}$ and $C_{F_{t^*}}$ for two infinitely renormalizable maps in the families $H_t$ and $F_t$ are conjugate by a $C^{1+\alpha}$-transformation.

In this paper we begin a study of hyperbolic sets associated with period-doubling bifurcations with the ultimate goal of demonstrating largeness and universality of the Hausdorff dimension of  these sets.

Existence and properties of the hyperbolic sets arising in saddle-center bifurcations in area-preserving maps and Hamiltonian systems have been extensively studied in the literature. Exponential splitting of stable and unstable leafs in these hyperbolic sets has been addressed in \cite{La1,La2,DR,DGJS,DGG1,DGG2,DGG3,Ge1,Ge2,GS,FS,LMS}. Furthermore, exponential smallness of the splitting  can be used together with the methods developed by P. Duarte \cite{Du1,Du2} to bound the thickness of the hyperbolic Cantor sets near the bifurcation, and, to eventually estimate the Hausdorff dimension of the Cantor sets.   This approach has been already used on several occasions in the setting of the restricted three body problem \cite{GK1,GK2,GMS}.

\begin{figure}
\centering
\vspace{2mm}       
\begin{tabular}{c c c}
{\includegraphics[height=45mm,width=45mm,angle=-90]{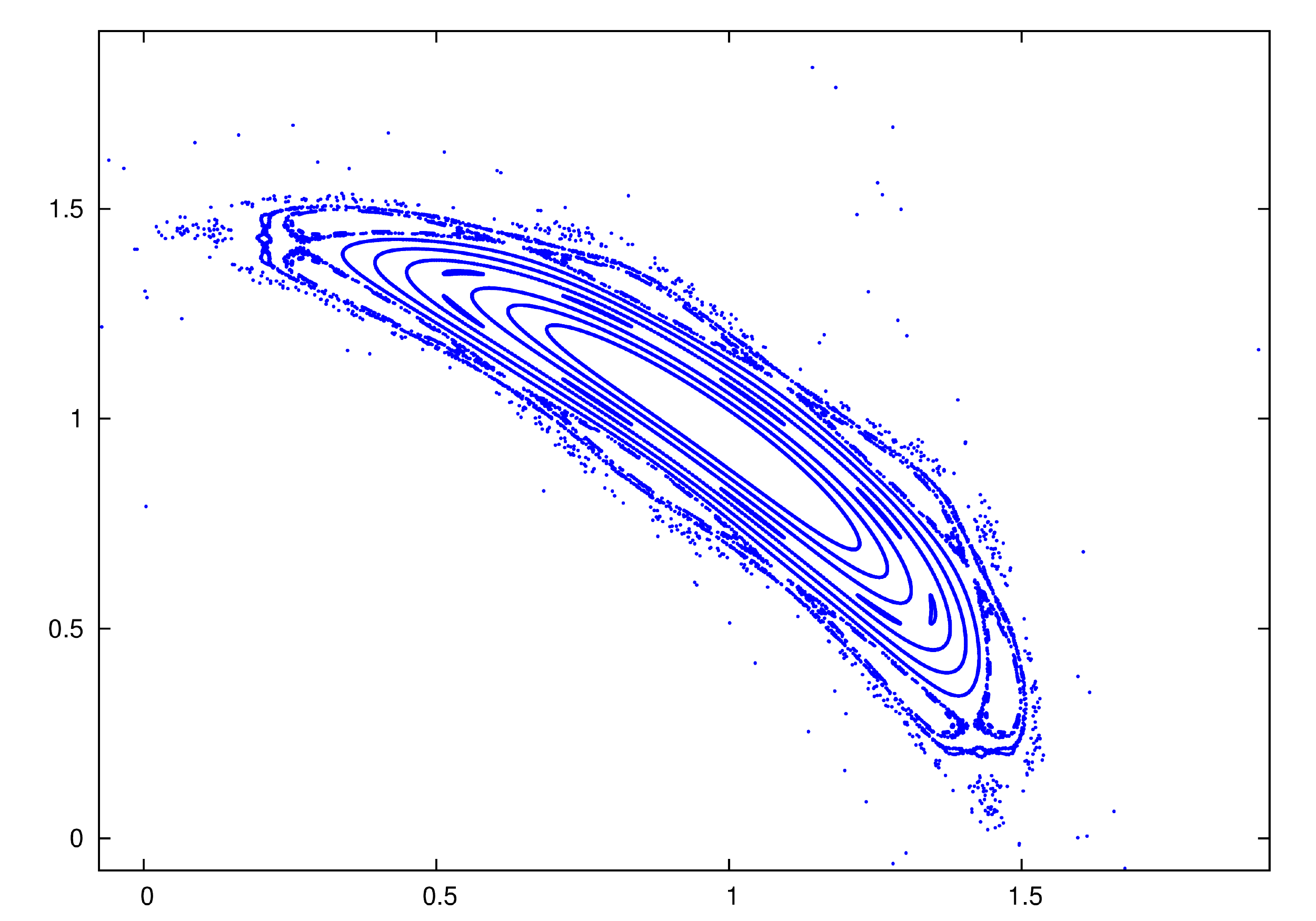}} & {\includegraphics[height=45mm,width=45mm,angle=-90]{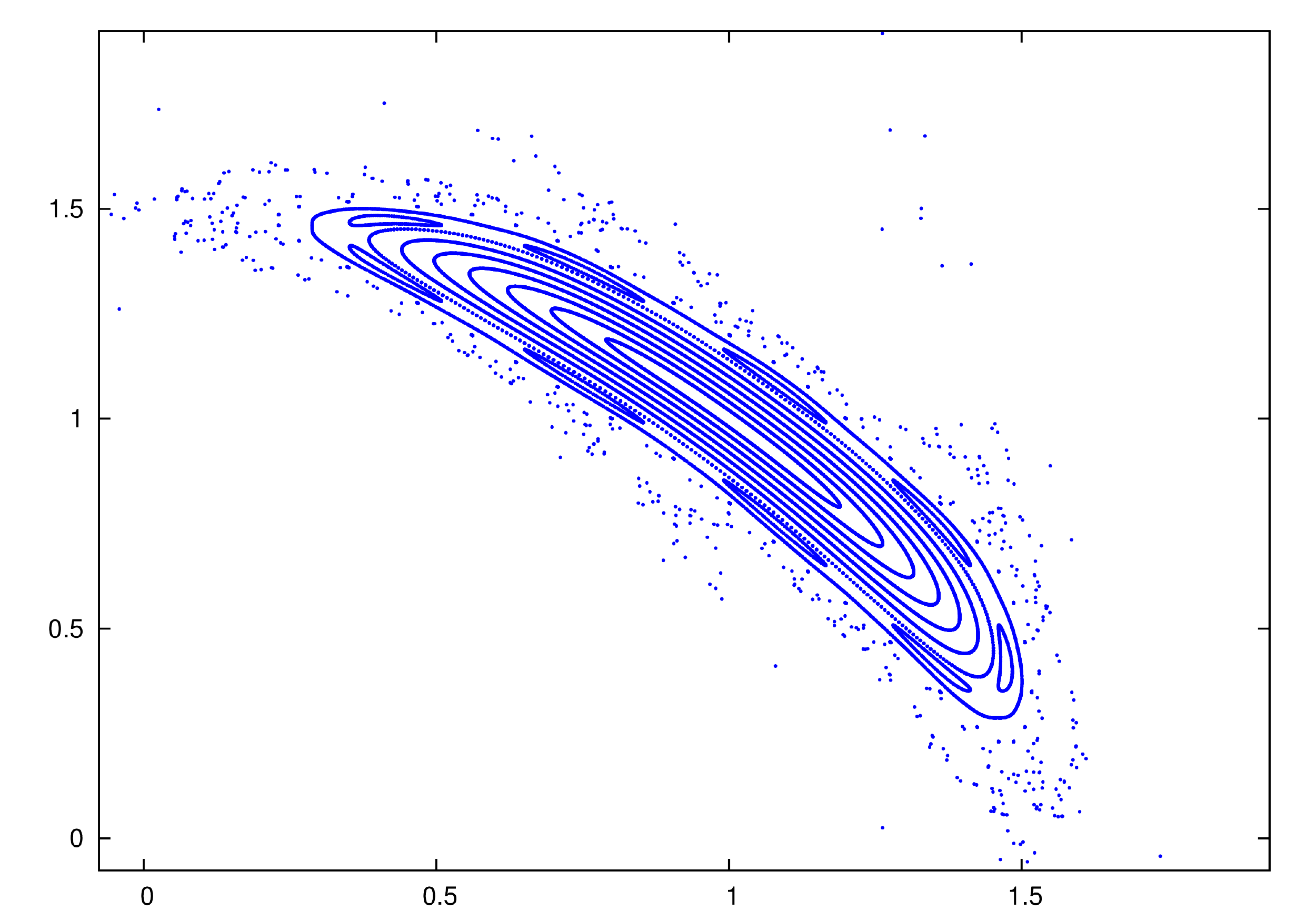}} &  {\includegraphics[height=45mm,width=45mm,angle=-90]{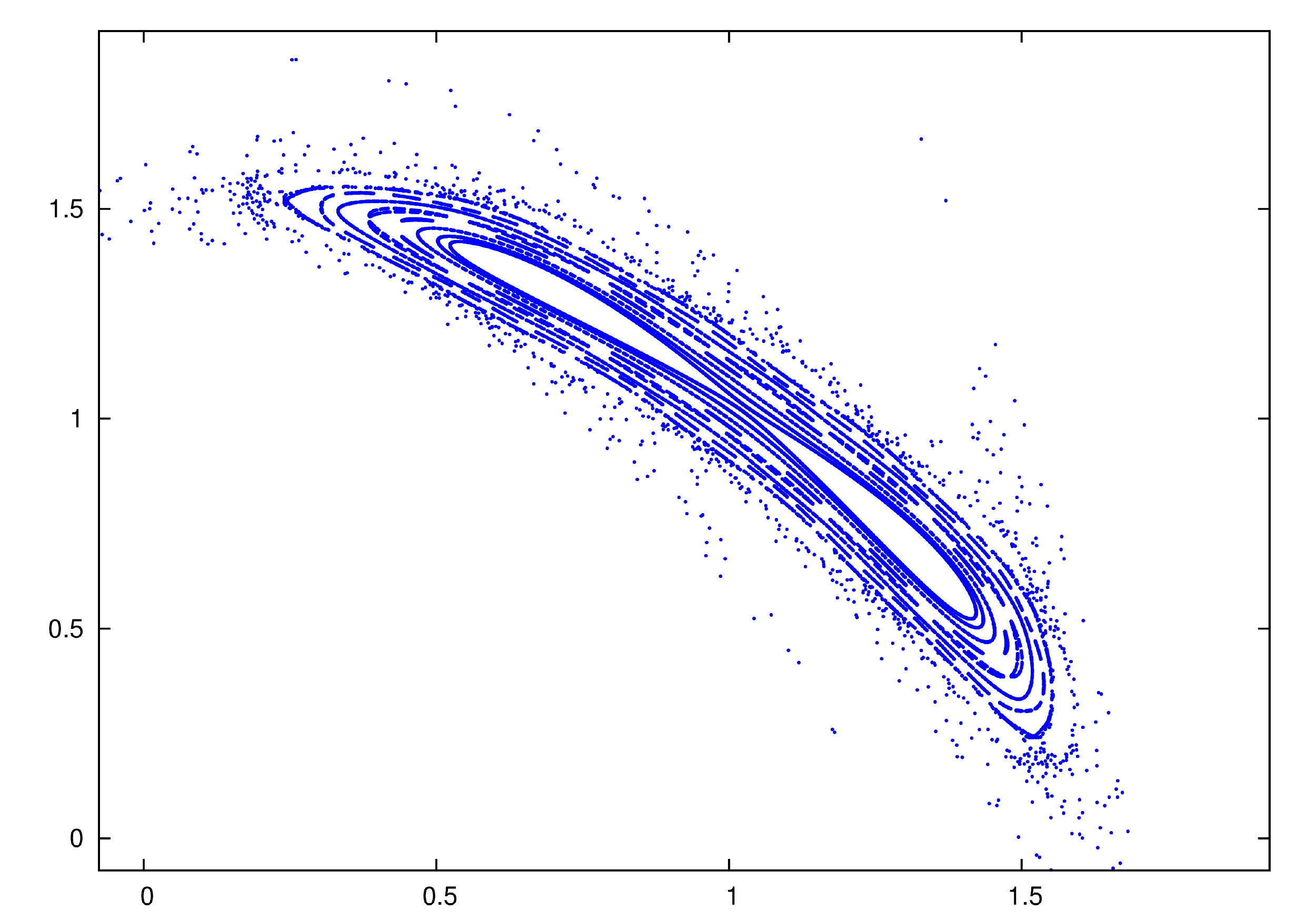}} \\
a) & b) & c)
\end{tabular} 
\caption{A period doubling bifurcation of the fixed point $(-1+\sqrt{4+\eps},-1+\sqrt{4+\eps})$ of the  H\'enon map $\He(x,y)=(y, -x +3 + \eps -y^2).$: a) $\eps=-0.1$, b) $\eps=0.0$, c) $\eps=0.1$.}
\label{bifurcation}
\end{figure}

We conjecture that the hyperbolic Cantor sets arising in homoclinic intersections in families of exact symplectic maps near a period-doubling bifurcation also approach the maximal Hausdorff dimension. Hyperbolicity of period-doubling renormalization causes this behaviour to be universal in typical families  - one-parameter families transversal to the codimension one renormalization stable manifold.

\begin{conj}
Let $t_k$ be an increasing convergent sequence of period-doubling bifurcation parameters in a typical family $H_t$ of exact symplectic diffeomorphisms of a subset of $\RR^2$ to $\RR^2$. 

Then there exists a sequence $\eps_k>0$ such that every map $H_t$ with $t \in (t_k,t_k+\eps_k)$ admits an invariant hyperbolic set $\Lambda_t$ whose Hausdorff dimension approaches the maximal,
$$HD(\Lambda_t) \ge  2- C (t-t_k)^\alpha,$$
where $C>0$ and $\alpha>0$ are  some universal constant.
\end{conj}

The proof of the conjecture will be split in three steps. 

In the first step we prove the exponential smallness of the splitting between the separatrices of the unstable periodic point nearing a period-doubling bifurcation. This will be subject of the present paper.

In the second step, one uses the exponential smallness of the splitting  to  estimate the Hausdorff dimension of the Cantor sets. 

The third step uses the hyperbolicity of renormalization to demonstrate that these estimates are universal for a class of typical families.

Our main result of the present paper will be an estimate on a splitting angle $\alpha$ between the separatrices $\gamma^\pm(\tau)$ in the homoclinic orbit $\gamma^+ \cap \gamma^-=\{ \gamma^-(\tau_i) \}_{i=0}^\infty=\{ \gamma^+(\tau_i) \}_{i=0}^\infty$:  
$$\sin \alpha = {C \over  \| \dot{\gamma}^+(\tau_i)\|  \| \dot{\gamma}^-(\tau_i)\|   } e^{-{\pi^2 \over h}}+ O\left( e^{-2  {\pi^2 \over  h}} \right),$$
where  $\|\dot{\gamma}^\pm(\tau_i)\|$ are lengths of the corresponding tangent vectors at the points of the homoclinic orbit, $C$ is some constant and $h=|\lambda_\pm+1|+O(h^2)$ is the deviation of the eigenvalues $\lambda_\pm$ of a saddle periodic point near the period-doubling bifurcation through $-1$.

\begin{figure}[!t]
\centering
\vspace{2mm}       
\begin{tabular}{c c c}
{\includegraphics[height=45mm,width=45mm,angle=-90]{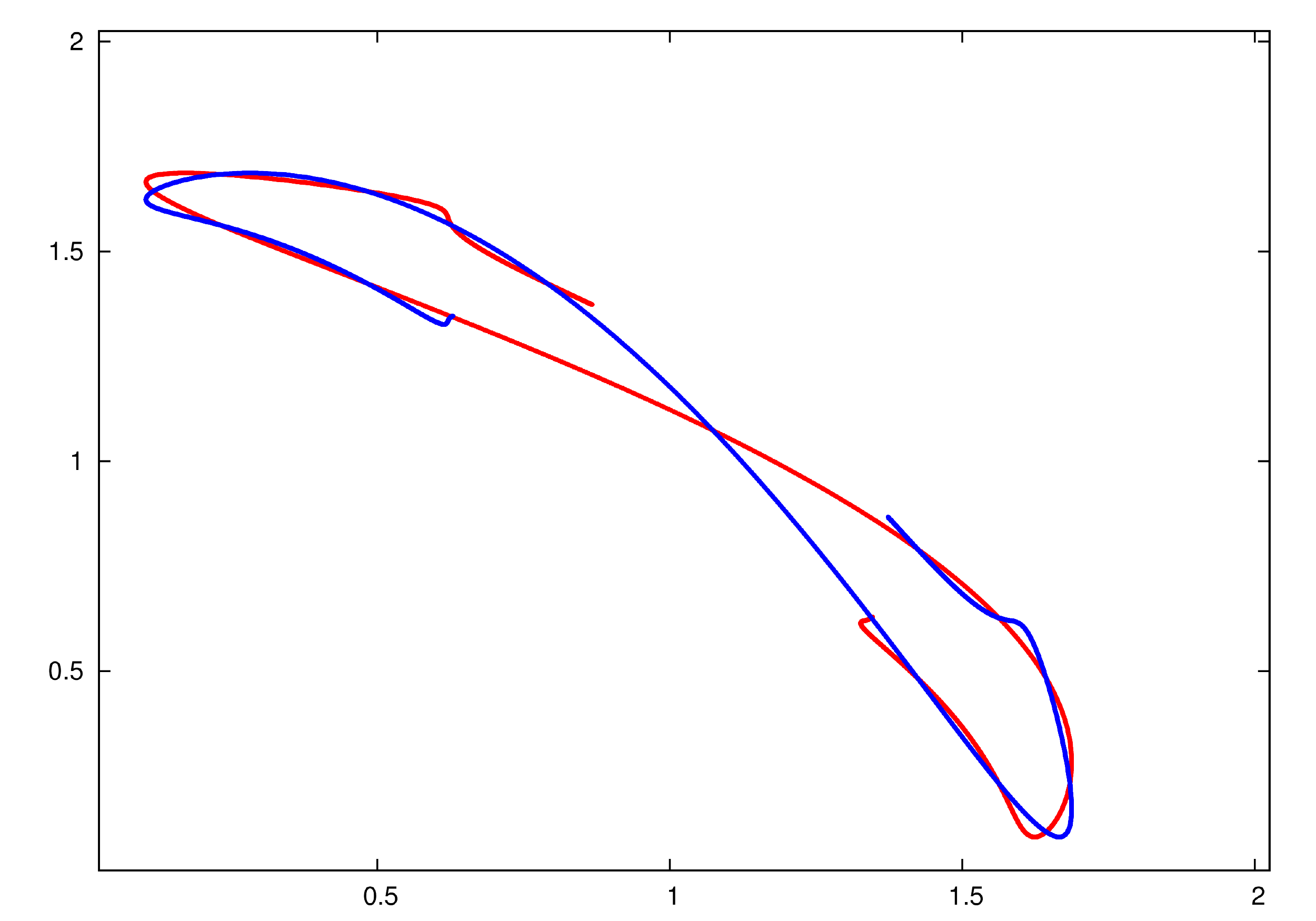}} & {\includegraphics[height=45mm,width=45mm,angle=-90]{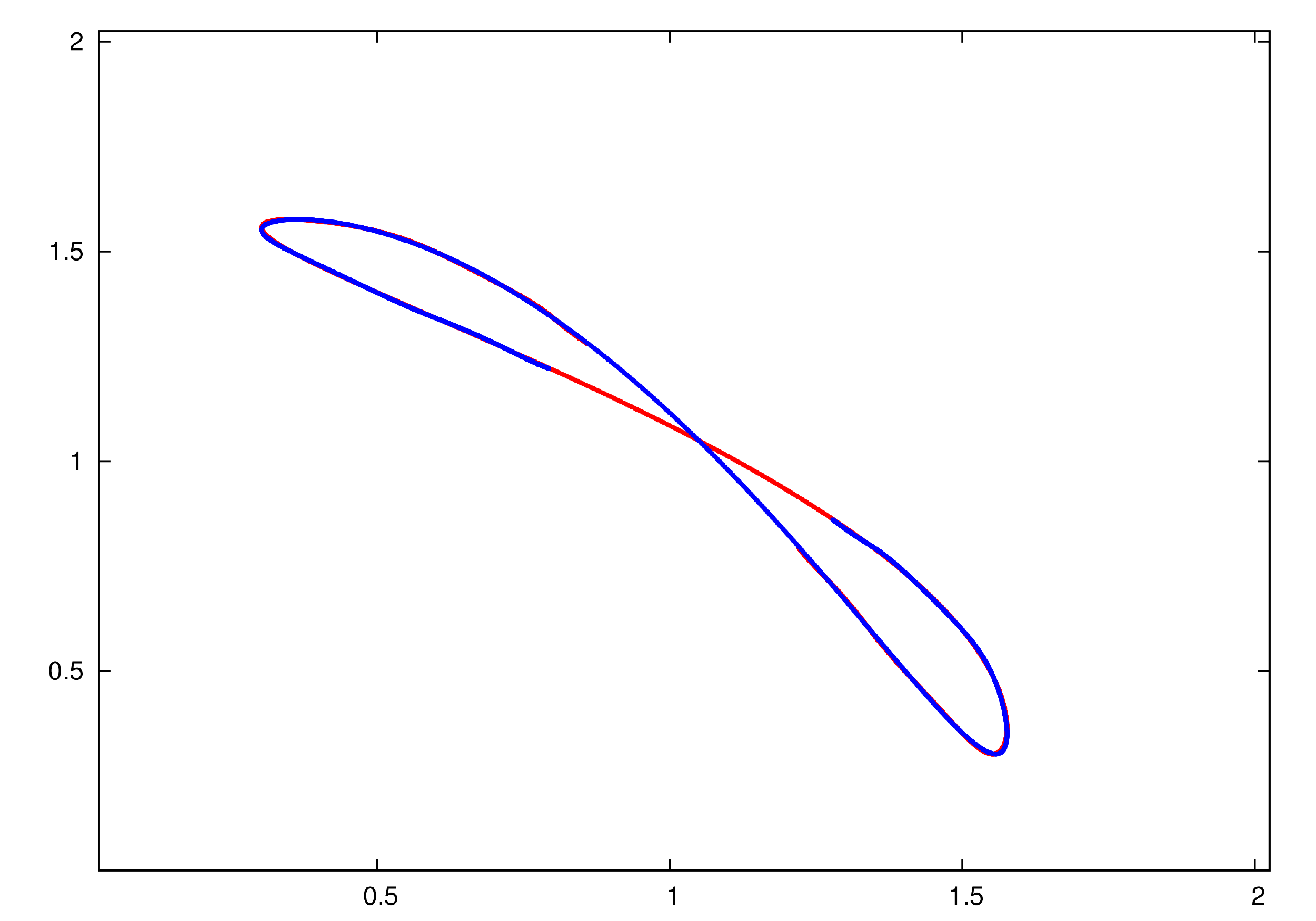}} &  {\includegraphics[height=45mm,width=45mm,angle=-90]{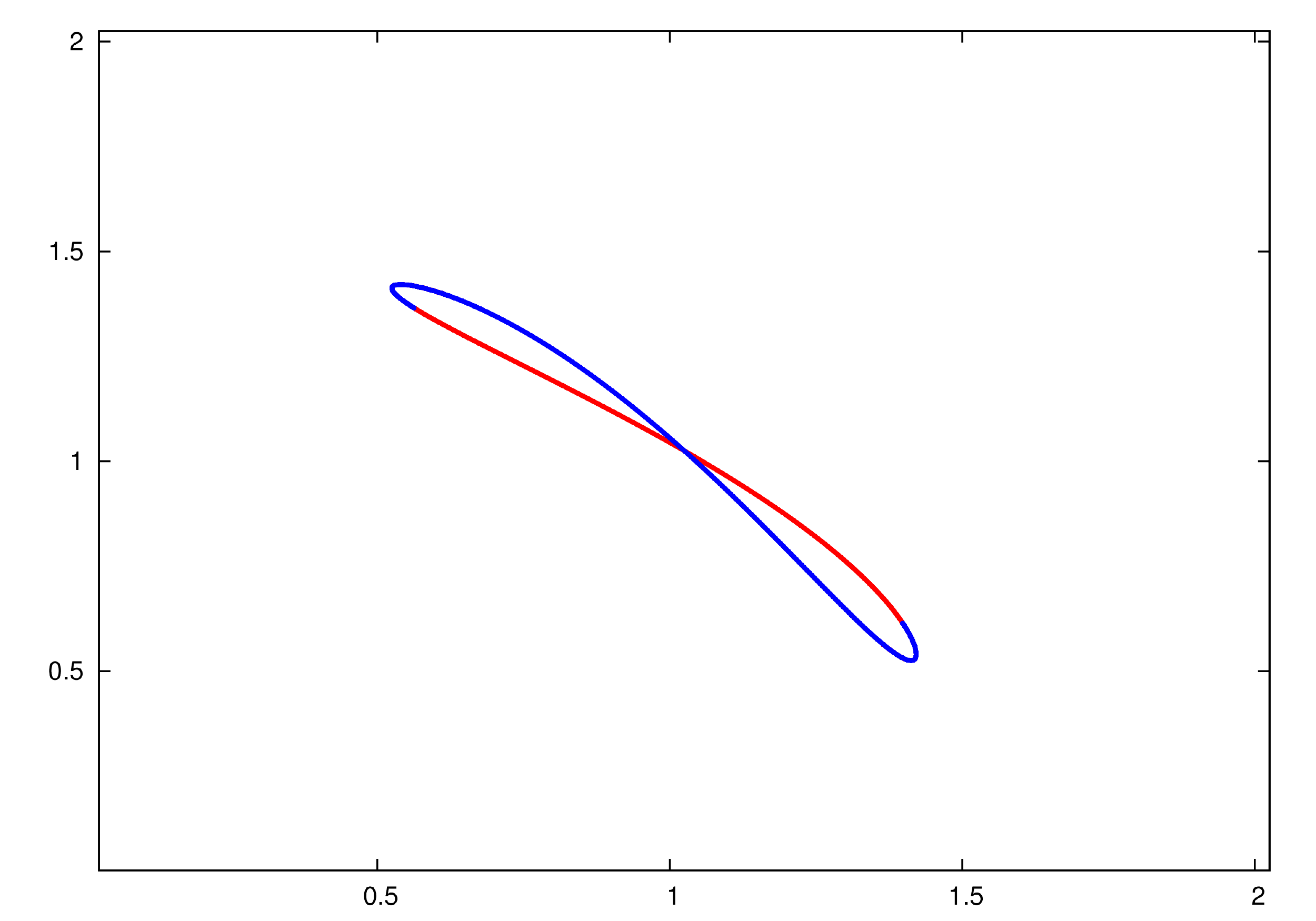}} \\
a) & b) & c)
\end{tabular} 
\caption{Stable and unstable manifolds at the fixed point $(-1+\sqrt{4+\eps},-1+\sqrt{4+\eps})$ of $\He$: a) $\eps=0.3$, b) $\eps=0.2$, c) $\eps=0.1$.}
\label{bifurcation}
\end{figure}

\section{Expansion of separatrices in a small parameter} \label{sec:expansion} 

\subsection{Expansion in the lowest order}

We consider the area-preserving H\'enon family $H_\eps$ close to the period-doubling bifurcation parameter $\eps=0$,
\begin{equation}
\He(x,y)=(y, -x +3 + \eps -y^2).
\end{equation}

For $\eps>0$, the map has a pair of stable nodes and a saddle point at $p_\eps=(x_\eps,y_\eps)$
\begin{equation} \label{eq:fp}
x_\eps=y_\eps=-1+ \sqrt{4+\eps}.
\end{equation}

We will now consider the second iterate of $\He$,
\begin{equation} \label{eq:H2}
\He^2(x,y)=(-x +3 + \eps -y^2,-y+3+\eps-(-x+3+\eps-y^2)^2).
\end{equation}

The second iterate has the following derivative at $(x_\eps,y_\eps)$.
\begin{equation}\label{eq:der}
D \He^2(x_\eps,y_\eps)=   \left[
\begin{tabular}{ c c}
  $-1 $    &                 $ -2 y_\eps$ \\
  $-2 y_\eps^2 + 2 \eps - 2 x_\eps + 6 $  &   $ -4 (y_\eps^2 - \eps + x_\eps - 3) y_\eps - 1 $ 
\end{tabular}
\right],
\end{equation}
and the following eigenvalues
\begin{eqnarray} \label{eq:evalues}
\nonumber \lambda_\pm &=& 9+2 \eps -4 \sqrt{ 4 + \eps  } \pm 2 \sqrt{ 36 +13 \eps - ( 4 \eps +18 ) \sqrt{4+\eps  }+\eps^2  }\\
\nonumber &=&1\pm\sqrt{2 \eps }+\eps\pm{ 9 \over 32} \sqrt{2} \eps^{3 \over 2} + o(\eps^{3 \over 2}).
\end{eqnarray}

The map $\He^2$ is reversible  through the involution $R \circ \He$, $R(x,y)=(y,x)$. 

Introduce a small parameter $h$,
\begin{equation}
h=\log \lambda_+=\sqrt{2 \eps} +o(\eps^{1 \over 2}) \implies \eps={1 \over 2} h^2 +{5 \over 192} h^4+{25\over 73728} h^6 + O(h^8).
\end{equation}

We will now pass to the coordinate system where the map $\He^2$ assumes the form $F_0+\eps F_1 + o(\eps)$, 
\begin{equation}
F_0(x,y)=(x+y,y)+O(x^i y ^j), \quad i+j \ge 2.
\end{equation}
 
 This is achieved through a conjugation by $T(x,y)=(y/2-1/2, x+y-2)$. The second iterate in the new coordinate system assumes the form
\begin{eqnarray}
\nonumber \Fe(x,y) \hspace{-3mm}&=&\hspace{-3mm} T(\He^2(T^{-1}(x,y)))=F_0(x,y)+\eps F_1(x,y) + \eps^2 F_2(x,y) + O(\eps^3), \\
\nonumber F_0(x,y)\hspace{-3mm}&=&\hspace{-3mm} (x\hspace{-0.6mm}+\hspace{-0.6mm}y\hspace{-0.6mm}+\hspace{-0.6mm}2 x^2\hspace{-1.5mm} -\hspace{-0.6mm}2 x y\hspace{-0.6mm} - \hspace{-0.6mm}{1 \over 2} y^2  \hspace{-1.5mm}- \hspace{-0.6mm}8 x^3\hspace{-1.5mm}-\hspace{-0.6mm} 4 x^2 y \hspace{-0.6mm}-\hspace{-0.6mm}8 x^4\hspace{-0.6mm}, \hspace{-0.3mm}y\hspace{-0.6mm}-\hspace{-0.6mm}4 x y\hspace{-0.6mm} -\hspace{-0.6mm}y^2\hspace{-1.5mm} -\hspace{-0.6mm}16 x^3\hspace{-1.5mm}-\hspace{-0.6mm}8 x^2 y\hspace{-0.6mm} -\hspace{-0.6mm}16 x^4 \hspace{-0.3mm}), \\
\nonumber F_1(x,y)\hspace{-3mm}&=&\hspace{-3mm} (2 x+y +4 x^2 -{1 \over 2}, 4 x +2 y + 8 x^2), \\
\nonumber F_2(x,y) \hspace{-3mm}&=&\hspace{-3mm} ( -{1 \over 2},-1).
\end{eqnarray}

We will also use the notation 
\begin{equation}
\label{eq:Ge} G_\eps=T \circ H_\eps \circ T^{-1}.
\end{equation}
Let $\gamma^{\pm}(t)=(x^\pm(t),y^\pm(t)$ be the parametrization of the stable ($+$) and unstable ($-$) manifolds of $w_\eps$,
\begin{equation}\label{eq:newfp}
w_\eps=T(p_\eps)=({1 \over 2}\sqrt{4 + \eps}-1, 2 \sqrt{4+\eps}-4 ),
\end{equation}
such that
\begin{equation} \label{eq:param}
\Fe(\gamma^\pm(t))=\gamma^\pm (t+h).
\end{equation}

Necessarily, 
\begin{equation} \label{eq:bnd1}
\lim_{t \rightarrow \pm \infty} \gamma^\pm(t)=w_\eps.
\end{equation}

We will momentarily assume that $\gamma^\pm$ intersect along a homoclinic orbit.  Fix one of these homoclinic points $q_0$. Then, we can parametrize the pieces of $\gamma^\pm$ bounded by $w_\eps$ and $q_0$ simultaneously by a curve $\gamma=(x,y): \RR \mapsto \RR^2$ so that  
\begin{equation} \label{eq:bnd2}
\gamma(0)=\gamma^+(0)=\gamma^-(0)=q_0,
\end{equation}
and
$$\gamma(t)=\gamma^+(t), \quad \gamma(-t)=\gamma^-(t), \quad t \ge 0.$$ 

By reversibility of $\Fe$, 
$$\gamma(-t)=S (\gamma(t)),$$
where 
$$S=T \circ  R \circ H_\eps \circ T^{-1}.$$

We will further assume that $\gamma(t)$ admits the following expansion in powers of $h$.
\begin{eqnarray}
\nonumber x(t)&=&h \sum_{k=0}^\infty X_k(t) h^k, \\
\nonumber y(t)&=&h^2 \sum_{k=0}^\infty Y_k(t) h^k.
\end{eqnarray}
Then, the equation  $\Fe(\gamma(t))=\gamma(t+h)$ together with boundary condition $(\ref{eq:bnd1})$ assumes the following form in the lowest and the next lowest  orders in $h$.
\begin{eqnarray}
\nonumber  \gamma(t+h)-\gamma(t)\hspace{-2mm}&=&\hspace{-3mm}\left( h^2 \left(Y_0(t)+2 X_0(t)^2-{1 \over 4} +o(h)  \right) , \right. \\
 \nonumber  \phantom{\gamma(t+h)-\gamma(t)}\hspace{-2mm}&\phantom{=}&\hspace{-3mm}\phantom{(} \left. h^3 \left(-4 X_0(t) Y_0(t)-16 X_0(t)^3 -16 X_0(t)^3 +2 X_0(t) +o(h) \right) \right), \\
\nonumber \lim_{t \rightarrow \pm \infty} (X_0(t),Y_0(t))\hspace{-2mm}&=&\hspace{-3mm}\left(0, {1  \over 4}\right) + o(h).
\end{eqnarray}
therefore,
\begin{eqnarray}
\label{eq:X0flow} X_0'(t) \hspace{-2mm}&=&\hspace{-2mm}Y_0(t)+2 X_0(t)^2-{1 \over 4} +o(h), \\
\label{eq:Y0flow} Y_0'(t)\hspace{-2mm}&=&\hspace{-3mm}-4 X_0(t) Y_0(t)\hspace{-1mm}-\hspace{-1mm}16 X_0(t)^3 \hspace{-1mm}-\hspace{-1mm}16 X_0(t)^3 \hspace{-1mm}+\hspace{-1mm}2 X_0(t)\hspace{-1mm}+\hspace{-1mm}o(h), \\
\label{eq:bnd3} \lim_{t \rightarrow \pm \infty} \hspace{-1mm}(X_0(t),Y_0(t))\hspace{-2mm}&=&\hspace{-2mm}\left(0, {1  \over 4}\right) +o(h).
\end{eqnarray}

The solution to the equations $(\ref{eq:X0flow})-(\ref{eq:bnd3})$ in the limit $h \rightarrow 0$  is as follows.

\begin{eqnarray}
\label{eq:X0} X_0(t) &=& {1 \over 2} \sech (t), \\
\label{eq:Y0} Y_0(t) &=& -{1 \over 2} \sech (t) \left( \tanh (t) +\sech (t) \right) +{1 \over 4}
\end{eqnarray}

\begin{figure}
\centering
\vspace{2mm}       
\includegraphics[height=50mm,width=45mm, angle=0]{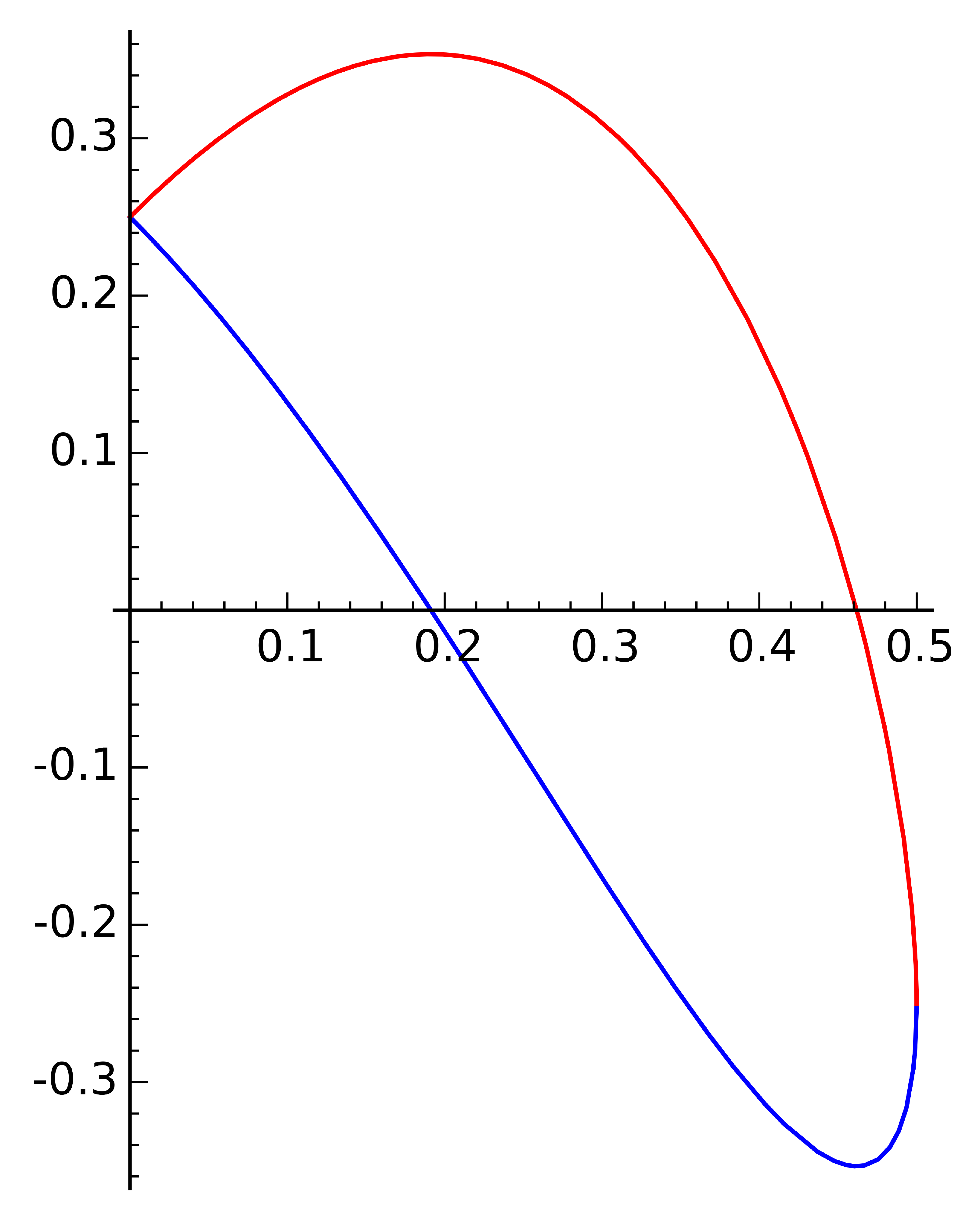} 
\caption{The curve $(X_0(t),Y_0(t))$ for $t \in \RR$. The restriction of the curve to positive $t$ is given in blue, to negative $t$ - in red.}
\label{fig:xy0}
\end{figure}

The equations~(\ref{eq:X0}) and~(\ref{eq:Y0}) demonstrate that the separatrices $\gamma^\pm(t)$, considered in the lowest order in $h$, constitute a homoclinic connection. The standard arguments imply that the separatrices, if shown to exist, must intersect along a homoclinic orbit $\{q_i\}_{i=0}^\infty$,
$$q_i=\gamma^+(t_i^+)=\gamma^-(t_i^-).$$

We define the {\it primary homoclinic point} $q_0$ to be the intersection with the smallest absolute values of the parametrizations $t^+_i$ and $t^-_i$ of the intersections.  

Let $\Omega=d x \wedge d y$ be the canonical symplectic form in $\RR^2$. Define the {\it homoclinic invariant} as 
$$\theta:= \Omega(\dot{\gamma}^+(t_0^-),\dot{\gamma}^+(t_0^+)).$$
Since $\Fe$ is an exact symplectic diffeomorphism, we have that
$$\theta= \Omega(\dot{\gamma}^+(t_i^-),\dot{\gamma}^+(t_i^+)), \ i \in \NN.$$

The homoclinic invariant is straightforwardly related to the angle $\alpha$  between the tangents to the separatrices at the homoclinic points:
$$\sin \alpha_i={ \theta \over \| \dot \gamma^+(t_i^+)\|  \| \dot \gamma^-(t_i^-)\|   }.$$

We are now ready to give a precise statement of our main theorem. Because of the translational freedom of the parametrization, we can assume w.l.o.g. that $t^+_0=t^-_0=0$.
\begin{mainthm}\label{mainthm}
If $h$ is sufficiently small, then there exists a constant $\Theta_1$, and for any positive $\kappa<1$, a constant $C>0$, such that homoclinic invariant satisfies
$$\left|\theta -\Theta_1 e^{-{\pi^2 \over h}} \right| < C e^{-2 (1-\kappa) {\pi^2 \over  h}}.$$
\end{mainthm}

\begin{remark}
At present we do not demonstrate that $|\Theta_1|>0$. Positivity of this constant would be  an important result, whose non-trivial proof is outside of the scope of this paper. 
\end{remark}

\subsection{Expansion in the higher orders}

Writing out the equation for $\gamma(t+h)-\gamma(t)$ in the next order in $h$ results in the following  equations for $X_1(t)$ and $Y_1(t))$ in the limit $h \rightarrow 0$:
\begin{eqnarray}
\label{eq:X1flow} X_1'(t) \hspace{-2mm} &=& \hspace{-2mm}  {1 \over 4} \sech(t) \hspace{-0.6mm}-\hspace{-0.6mm}{1 \over 2} \sech^3(t) \hspace{-0.6mm}+\hspace{-0.6mm} {1 \over 2} \sech^2(t) \tanh(t) \hspace{-0.6mm}+ \hspace{-0.6mm}2 \sech(t) X_1(t)  \hspace{-0.6mm}+ \hspace{-0.6mm}Y_1(t), \\
\nonumber   Y_1'(t)\hspace{-2mm}&=&\hspace{-2mm}-{1 \over 4} \sech^4(t) - {1 \over 4} \sech(t)^2 \tanh^2(t) + {1 \over 4} \sech^2(t) + {1 \over 4} (2 \sech^3(t) -  \\
\label{eq:Y1flow} \hspace{-2mm}&\phantom{=}& \hspace{-2mm}10 \sech^2(t) - 2 \sech(t) \tanh(t) - 1) X_1(t) - 2 \sech(t) Y_1(t) - {1 \over 16}, 
\end{eqnarray}
and
$$\lim_{t \rightarrow \pm \infty} (X_1(t),Y_1(t))=\left({1 \over 16}, 0 \right).$$

This system $ (\ref{eq:X1flow})-(\ref{eq:Y1flow})$ is equivalent to the following second order-differential equation
\begin{eqnarray}
\nonumber X_1''(t)&=&(1-6 \sech^2(t)) X_1(t)-{1 \over 16}+\\
\label{eq:X1dd} &\phantom{=}& {1 \over 2} \sech^4(t) -{1 \over 2} \sech^2(t)-{1 \over 2} (\sech(t)-6 \sech^3(t)) \tanh(t).
\end{eqnarray}

According to the result of \cite{Yo}, the solution of an initial value problem for the equation $(\ref{eq:X1dd})$ either approaches $1/16+C_1 e^{t} + C_2 e^{-t}$, for some constants $C_1 \ne 0$ and $C_2$:
$$\lim_{t \rightarrow  \infty} \left|{X_1(t)-{1 \over 16}  \over C_1 e^{t} + C_2 e^{-t} } - 1 \right|=0,$$
or $1/16+C_2 e^{-t}$, for some constant $C_2$:
$$\lim_{t \rightarrow \infty} \left|X_1(t) -{1 \over 16} - C_2 e^{-t} \right|=0,$$
and, furthermore, it  approaches $1/16+C_1 e^{t} + C_2 e^{-t}$, for some constants $C_1$ and $C_2 \ne 0$ as $t \rightarrow - \infty$,
$$\lim_{t \rightarrow  -\infty} \left|{X_1(t)-{1 \over 16}  \over C_1 e^{t} + C_2 e^{-t} } - 1 \right|=0,$$
or $1/16+C_1 e^{t}$, for some constant  $C_1$:
$$\lim_{t \rightarrow -\infty} \left|X_1(t) -{1 \over 16} - C_1 e^{t} \right|=0,$$

The two solutions of $(\ref{eq:X1dd})$  satisfying the boundary condition $\lim_{t \rightarrow \pm \infty} (X_1(t),Y_1(t))=\left({1 / 16}, 0 \right)$ have the following asymptotic form
\begin{eqnarray}
\nonumber &&X_1^+(t) \rightarrow {1 \over 16} +C_+ e^{-t}, \ {\rm as} \ t\rightarrow \infty, \\
\nonumber &&X_1^-(t) \rightarrow {1 \over 16} +C_- e^{t}, \ {\rm as} \ t \rightarrow -\infty. 
\end{eqnarray}


In general, we have in the limit $h\to 0$
\begin{eqnarray}
\label{eq:Xkflow} X_k'(t) &=&4 X_0 X_k +Y_k + g^{(1)}_k(X_0,Y_0,\ldots, X_{k-1}, Y_{k-1}),
\\
\label{eq:Ykflow} Y_k'(t)&=& (2- 4Y_0 - 48 X_0^2) X_k - 4 X_0 Y_k + g^{(2)}_k(X_0,Y_0,\ldots, X_{k-1}, Y_{k-1}), 
\end{eqnarray}
where $g^{(1)}_k$ and $g^{(2)}_k$ are polynomial functions of their variables.

The solution $(\ref{eq:X0})-(\ref{eq:Y0})$ has a singularity at $i \pi /2 + \pi n$. More generally,

\begin{lem} 
Solutions to an initial value problem for the  equations $(\ref{eq:Xkflow})-(\ref{eq:Ykflow})$ continue analytically to the strip $|\Im(t)|< \pi/2$, and have a simple pole at $t=i \pi /2 + \pi n$, $n \in \ZZ$. 
\end{lem} 
\begin{proof}
The system $(\ref{eq:Xkflow})-(\ref{eq:Ykflow})$ is linear in $(X_k,Y_k)$. The vector field of this planar system is a Lipschitz continuous function of $X_k$ and $Y_k$: indeed, the matrix valued function 
$$  
\left[4 X_0(t)  \hspace{33mm}  1  \atop   2- 4Y_0(t) - 48 X_0(t)^2  \hspace{7mm} - 4 X_0(t)   \right],
$$
has a bounded norm on $|\Im(t)|\leq \rho <\pi/2-\eps$ for any $\eps>0$ (see \ref{eq:X0}--\ref{eq:Y0}). By Picard-Lindel\"of existence theorem, the solutions extend analytically to $|\Im(t)|\leq \rho <\pi/2-\eps$, and since $\eps$ is arbitrary small, to $|\Im(t)|\leq \rho <\pi/2$. 
\end{proof}

\section{Existence of separatrices} \label{sec:existence} 

We will study the behaviour of the solutions $\gamma^\pm(t)$ in a neighborhood of one of the two nearest singularities, specifically, $t=i \pi /2$, by reparametrizing the curves $\gamma^\pm$ as follows
\begin{equation}
\alpha_\eps^\pm(\tau)=(x^\pm(t),y^\pm(t)), \ \tau={t -i {\pi \over 2} \over h}.
\end{equation}

The equation $(\ref{eq:param})$ together with the boundary condition $(\ref{eq:bnd1})$ assume the following form
\begin{eqnarray}
\label{eq:alphaeq} \Fe(\alpha_\eps^\pm(\tau))&=&\alpha_\eps^\pm (\tau+1) \\
\label{eq:bndalpha} \lim_{\tau \rightarrow \pm \infty} \alpha_\eps^\pm(\tau)&=& w_\eps.
\end{eqnarray}

We will look for $\alpha_\eps^\pm$ as a formal power series 
\begin{equation}
\alpha_\eps^\pm(\tau)=\sum_{k=0}^\infty \alpha_k^\pm(\tau) h^k. 
\end{equation}
Specifically, the equations $(\ref{eq:alphaeq})-(\ref{eq:bndalpha})$ in the lowest order in $h$ become
\begin{eqnarray}
\label{eq:alphaeq0} F_0(\alpha_0^\pm(\tau))&=&\alpha_0^\pm (\tau+1) \\
\label{eq:bndalpha0} \lim_{\tau \rightarrow \pm \infty} \alpha_0^\pm(\tau)&=&(0,0).
\end{eqnarray}

Our immediate goal will be to obtain a result about the existence of solutions to the equations $(\ref{eq:alphaeq})-(\ref{eq:bndalpha})$ and $(\ref{eq:alphaeq0})-(\ref{eq:bndalpha0})$. To that end, set formally,
$$\alpha_0^\pm(\tau) = \sum_{k=0}^\infty {\bf c}_k^\pm \tau^{-k}, \quad {\bf c}_k^\pm=(a_k^\pm,b_k^\pm).$$

A substitution of the power series in the equation $(\ref{eq:alphaeq0})$ allows to find several first coefficients in this expansion. We, therefore,  set
\begin{equation}
\label{eq:alpham0} \alpha_0^-(\tau) =  \left( {a_1^- \over \tau} + {a_2^-+\xi_0^-(\tau) \over \tau^2}, {b_2^- \over \tau^2} + {b_3^-+\eta_0^-(\tau) \over \tau^3}  \right).
\end{equation}
where $\xi_0^-(\tau)$, $\eta_0^-(\tau)=O(1)$, and
\begin{equation}
\label{eq:coeffab} a_1^-=\pm {i \over 2}, \quad b_2^-=\mp{i \over 2}+{1 \over 2}, \quad b_3^-=(\mp 2 i-2) a_2^-+{1 \over 2},
\end{equation}
$a_2^-$ being a free parameter. Similarly,
\begin{equation}
\nonumber \alpha_\eps^-(\tau)  =  \left( {a_1^- \over \tau} + {a_2^-+\xi_\eps^-(\tau) \over \tau^2}, {b_2^- \over \tau^2} + {b_3^-+\eta_\eps^-(\tau) \over \tau^3}  \right).
\end{equation}
We would like to remark that there are two possible choice of signs in the above coefficients. They correspond to two different unstable separatrices, each invariant under $\Fe^2$. The union of the two is the unstable separatrix, invariant under $\Fe$. We fix a choice of upper signs and proceed. 

Most of our computations below will be performed simultaneously for $\alpha^-_0$ and $\alpha_\eps^-$, to streamline the notation, we will use the shorthand 
\begin{equation}
\label{alphaomega} \alpha^\pm_\omega(\tau)=(x_\omega(\tau),y_\omega(\tau)),
\end{equation}
with $\omega \in \{0,\eps\}$.

We substitute this ansatz into $(\ref{eq:alphaeq})$ and $(\ref{eq:alphaeq0})$ to obtain an equation for $\xi_\omega^-$ and $\eta_\omega^-$, $\omega \in \{0,\eps\}$
\begin{eqnarray}
\nonumber    \xi_\omega^-(\tau+1)&=&\xi_\omega^-(\tau) (1+(2 i +2) \tau^{-1})+\eta_\omega^- \tau^{-1} \\
\label{eq:xit} &\phantom{=}& \phantom{\xi_\omega^-(\tau) (1+(2 i +2) \tau^{-1}) } +\tau^{-2} p_\omega(\tau^{-1}, \xi_\omega^-(\tau), \eta_\omega^-(\tau) ),   \\
\nonumber  \eta_\omega^-(\tau+1)&=&\xi_\omega^-(\tau) (2 i +10) \tau^{-1}+(1+(3- 2 i) \tau^{-1}) \eta_\omega^-(\tau)  \\
\label{eq:etat} &\phantom{}& \phantom{ \xi_\omega^-(\tau) (2 i +10) \tau^{-1}}+ \tau^{-2} q_\omega(\tau^{-1},\xi_\omega^-(\tau),\eta_\omega^-(\tau)),
\end{eqnarray}
where $p_\omega$ and $q_\omega$ are polynomials of their arguments of order $0$.


Let $\delta \in (0, \pi/2)$ and $A > 0$.  Set 
 \begin{eqnarray}
\label{eq:DA} \cD_{A,\delta}^-&=&\{z \in \CC: \left| \arg \left(z+A  \right) \right|  \ge \delta \},
 \end{eqnarray}
and define $\cD_{A,\delta}^+$ to be the reflection of this set with respect to the $y$-axes.

 Furthermore, let $\cD$ be any closed domain in $\CC$ such that $0 \notin \cD$. Given non-negative $\mu$, denote by $\cX_\mu(\cD)$ the Banach space of complex valued continuous functions in $\cD$ analytic in ${\rm int} \ \cD$, for which the following norm is finite,
 $$\| f \|_\mu =\sup_{z \in \cD} |z^\mu f(z)|.$$


We are now ready to prove the existence of solutions of the equations $(\ref{eq:xit})-(\ref{eq:etat})$. We will prove the existence for both equations  $(\ref{eq:alphaeq})-(\ref{eq:bndalpha})$ and  $(\ref{eq:alphaeq0})-(\ref{eq:bndalpha0})$ simultaneously by employing the subscript $\omega \in \{0,\eps\}$ in our notation (e.g.  $(\xi_\omega^-, \eta_\omega^-)$).

\begin{prop} \label{prop:unst_existence}
Let $\omega \in \{0,\eps\}$, and let $\delta>0$  be fixed. Then, for every $\rho>0$ and $0<\vareps<1$ there exists $A>0$ and a function  $(\xi_\omega^-,\eta_\omega^-)$ which is the unique solution of the equations $(\ref{eq:xit})-(\ref{eq:etat})$  in $\cB_\rho(0) \subset \cX_{1-\vareps}(\cD_{A,\delta}^-) \times \cX_{1-\vareps}(\cD_{A,\delta}^-)$.
\end{prop}
\begin{proof}
Fix $\rho>0$, $\mu>0$ and $A>0$. Fix a pair $\cX_{\mu}(\cD_{A,\delta}^-) \times \cX_{\mu}(\cD_{A,\delta}^-)$, and consider the functions 
$$f_\omega(\tau)=p_\omega(\tau^{-1},\xi_\omega(\tau),\eta_\omega(\tau))$$ 
and
$$g_\omega(\tau)=q_\omega(\tau^{-1},\xi_\omega(\tau),\eta_\omega(\tau)).$$ 
We would like to find a particular solution to the non-homogeneous linear system of difference equations
\begin{eqnarray}
\nonumber    \xi_\omega^-(\tau+1)&=&\xi_\omega^-(\tau) (1+(2 i +2) \tau^{-1})+\eta_\omega^- \tau^{-1} +\tau^{-2} f_\omega(\tau),\\
\nonumber  \eta_\omega^-(\tau+1)&=&\xi_\omega^-(\tau) (2 i +10) \tau^{-1}+(1+(3- 2 i) \tau^{-1} ) \eta_\omega^-(\tau) + \tau^{-2} g_\omega(\tau).
\end{eqnarray}
To that end, we first diagonalize this system
\begin{eqnarray}
\label{eq:hxi}    \hat \xi_\omega^-(\tau+1)&=&\hat{\xi}_\omega^- +\tau^{-2} \hat{f}_\omega(\tau,\hat \xi_\omega, \hat \eta_\omega)=\hat{\xi}_\omega^- +\tau^{-2} \hat{f}_\omega(\tau),\\
\label{eq:heta}   \hat \eta_\omega^-(\tau+1)&=&\left( 1+{5 \over \tau} \right) \hat \eta_\omega^- + \tau^{-2} \hat{g}_\omega(\tau,\hat \xi_\omega, \hat \eta_\omega)=\left( 1+{5 \over \tau} \right) \hat \eta_\omega^- + \tau^{-2} \hat g_\omega(\tau),
\end{eqnarray}
where $(\hat \xi_\omega, \hat \eta_\omega)=L^{-1}(\xi_\omega,\eta_\omega)$ and 
$$L=
\left[
\begin{matrix}
1 & 1 \\
-2-2 i & 3-2 i
\end{matrix}
\right].
$$
is the diagonalizing coordinate change.

It is elementary to check that a particular solution to $(\ref{eq:hxi})$-$(\ref{eq:heta})$ is given formally by 
\begin{equation}
\nonumber  \left[ \hat \xi_\omega^-(\tau) \atop \hat \eta_\omega^-(\tau) \right]= \sum_{k=1}^\infty \prod_{i=1}^{k-1} A(\tau-i) \left[  (\tau-k)^{-2} \hat f_\omega (\tau-i) \atop  (\tau-k)^{-2} \hat g_\omega (\tau-k) \right] ,
\end{equation}
where
$$A(\tau)=
\left[
\begin{matrix}
1 & 0 \\
0 & 1+{5 \over \tau}
\end{matrix}
\right].
$$
Notice, that $|1+5/\tau|<1$ for all $\tau$ such that $\Re(\tau)<-5/2$. Therefore, the  product  $\prod_{i=1}^{\infty} A(\tau-i)$ converges to
$$B=
\left[
\begin{matrix}
1 & 0 \\
0 & 0
\end{matrix}
\right].
$$
for all $\tau \in \cD_{A,\delta}^-$ since all but a finite number $(\tau-i)$, $i \in \NN$, satisfy $\Re{(\tau-i)} < -5/2$. In fact, a straightforward calculation shows that
$$\prod_{i=1}^{k-1} A(\tau-i)= \left[
\begin{matrix}
1 & 0 \\
0 &  \prod_{m=0}^4 {\tau+m \over \tau-(k-1)+m }
\end{matrix}
\right].
$$
A look at the geometry of the set $D^-_{A,\delta}$ shows that the absolute value of every factor in  $\prod_{m=0}^4 {\tau+m \over \tau-(k-1)+m }$ is bounded from above by $\sin^{-1} \delta$. Therefore, 
$$ \left| \prod_{m=0}^4 {\tau+m \over \tau-(k-1)+m } \right| \le \sin^{-5} \delta.$$

Also, notice, that $|\tau-k| \ge (A+k) \sin \delta$. Therefore,
\begin{eqnarray}
\nonumber \|\hat \xi_\omega^-  \|_\mu &\le& \|\tau^{-1+\vareps} \hat f_\omega \|_\mu \sum_{k=1}^\infty {1 \over |A+k|^{1+\vareps} \sin^{1+\vareps} \delta}= { \zeta(1+\vareps,A)  \over \sin^{1+\vareps} \delta} \|\tau^{-1+\vareps} \hat f_\omega \|_\mu,, \\
\nonumber \|\hat \eta_\omega^- \|_\mu &\le& \|\tau^{-1+\vareps} \hat g_\omega \|_\mu \sum_{k=1}^\infty {1 \over |A+k|^{1+\vareps} \sin^{6+\vareps} \delta}={ \zeta(1+\vareps,A)  \over \sin^{6+\vareps} \delta} \|\tau^{-1+\vareps} \hat g_\omega \|_\mu,
\end{eqnarray}
where $\zeta$ is the Hurwitz zeta function. 

Since the polynomial $(p_\omega,q_\omega)$  does contain a constant term, so does $(\hat p_\omega, \hat q_\omega)$, and it follows that $\tau^{-1+\vareps} \hat f_\omega$ and $\tau^{-1+\vareps} \hat g_\omega$ are in $\cX_{1-\vareps} (\cD_{A,\delta}^-)$ whenever $\hat \xi_\omega$ and  $\hat \eta_\omega$ are in  $\cX_\mu(\cD_{A,\delta}^-)$ for any $\mu \ge 0$. Furthermore, since $\hat p_\omega$ and $\hat q_\omega$ are polynomials, the norms  $\| \tau^{-1+\vareps} \hat f_\omega \|_{1-\vareps}$ and $\| \tau^{-1+\vareps} \hat g_\omega \|_{1-\vareps}$ are  bounded by some constant depending on  $\hat \rho$  whenever $( \hat \xi_\omega, \hat \eta_\omega) \in  \cB_{\hat \rho} \subset \cX_{1 - \vareps} (\cD_{A,\delta}^-) \times \cX_{1-\vareps}(\cD_{A,\delta}^-)$. The map
$$T(\hat \xi_\omega,\hat \eta_\omega) = \sum_{k=1}^\infty \prod_{i=1}^{k-1} A(\tau-i) \left[  (\tau-k)^{-2} \hat f_\omega (\tau-i) \atop  (\tau-k)^{-2} \hat g_\omega (\tau-k) \right]$$
is an analytic map from $\cB_{\hat \rho} \subset \cX_{1-\vareps} (\cD_{A,\delta}^-) \times \cX_{1-\vareps }\cD_{A,\delta}^-)$ to $\cX_{1-\vareps}(\cD_{A,\delta}^-) \times \cX_{1-\vareps}(\cD_{A,\delta}^-)$  whose norm  admits a bound
$$\| T(\hat \xi_\omega,\hat \eta_\omega) \|_{1-\vareps} \le  C_{\hat \rho}{\zeta(1+\vareps,A) \over  \sin^{6+\vareps} \delta}  \|(\hat \xi_\omega,\hat \eta_\omega) \|_{1-\vareps}.$$
We recall, that $\zeta(1+\vareps,A)$ is a monotone decreasing function of its second argument on $\RR_+$, with $\lim_{x \rightarrow \infty}  \zeta(1+\vareps,A)=0$, therefore $T$ is a map of $\cB_{\hat \rho} \subset \cX_{1-\vareps}(\cD_{A,\delta}^-) \times \cX_{1-\vareps} \cD_{A,\delta}^-)$ into itself if $A$ is chosen large enough.

 Given two functions $(\hat \xi_\omega,\hat \eta_\omega)$ and $(\widetilde{\hat \xi_\omega},\widetilde{\hat \eta_\omega})$ in $\cB_{\hat \rho} \subset \cX_{1-\vareps}(\cD_{A,\delta}^-) \times \cX_{1-\vareps}(\cD_{A,\delta}^-)$, the norm of difference of the action of $T$ is bounded as follows
$$\|T(\hat \xi_\omega,\hat \eta_\omega) - T(\widetilde{\hat \xi_\omega},\widetilde{\hat \eta_\omega}) \|_{1-\vareps} \le  {\zeta(1+\vareps,A) \over  \sin^{6+\vareps} \delta} \|( \hat f_\omega-\widetilde{\hat f_\omega}, \hat g_\omega-\widetilde{\hat g_\omega}) \|_{1-\vareps}.$$
The polynomials $\hat q_\omega$ and $\hat p_\omega$ have explicit, but cumbersome equations: a straightforward computation demonstrates that 
\begin{eqnarray}
\nonumber \hat p_\omega(\tau^{-1},\hat \xi_\omega, \hat \eta_\omega) -\hat p_\omega(\tau^{-1},\widetilde{\hat \xi_\omega}, \widetilde{\hat \eta_\omega})&=&( \hat \xi_\omega- \widetilde{\hat \xi_\omega}) p^1_\omega + ( \hat \eta_\omega- \widetilde{\hat \eta_\omega}) p^2_\omega \\
 \nonumber \hat q_\omega(\tau^{-1},\hat \xi_\omega, \hat \eta_\omega) -\hat q_\omega(\tau^{-1},\widetilde{\hat \xi_\omega}, \widetilde{\hat \eta_\omega})&=&( \hat \xi_\omega- \widetilde{\hat \xi_\omega}) q^1_\omega + ( \hat \eta_\omega- \widetilde{\hat \eta_\omega}) q^2_\omega,
\end{eqnarray}
where $p^i_\omega(\tau^{-1},\hat \xi_\omega, \widetilde{\hat \xi_\omega},\hat \eta_\omega, \widetilde{\hat \eta_\omega})$ and $q^i_\omega(\tau^{-1},\hat \xi_\omega, \widetilde{\hat \xi_\omega},\hat \eta_\omega, \widetilde{\hat \eta_\omega})$ are some polynomials whose norm on $\cD_{A,\delta}^- \times \cB_{\hat \rho}^{\otimes 4}$ is bounded by some constant depending on $\hat \rho$.  We have, therefore,
$$\|T(\hat \xi_\omega,\hat \eta_\omega) - T(\widetilde{\hat \xi_\omega},\widetilde{\hat \eta_\omega}) \|_{1-\vareps} \le  K_{\hat \rho}{\zeta(1+\vareps,A) \over  \sin^{6+\vareps} \delta}  \|(\hat \xi_\omega,\hat \eta_\omega) - (\widetilde{\hat \xi_\omega},\widetilde{\hat \eta_\omega}) \|_{1-\vareps}.$$
Thus, $T$ is a metric contraction if $A$ is sufficiently large. The unique fixed point of this contraction in $\cB_{\hat \rho}$ is the particular solution of $(\ref{eq:hxi})$-$(\ref{eq:heta})$ that we are looking for.
\end{proof}

The reversibility of the map $\Fe$ leads to a similar existence result for the stable separatrix.

\begin{prop} \label{prop:st_existence}
Let $\omega \in \{0,\eps\}$, and let $\delta>0$  be fixed. Then, for every $\rho>0$ and $\vareps>0$  there exists $A>0$ and a function  $(\xi_\omega^+,\eta_\omega^+)$ which is the unique solution of the equations $(\ref{eq:xit})-(\ref{eq:etat})$  in $\cB_\rho(0) \subset \cX_{1-\vareps}(\cD_{A,\delta}^+) \times \cX_{1-\vareps}(\cD_{A,\delta}^+)$.

\end{prop}
\begin{proof}
The map $\Fe$ is reversible under the involution $S$, therefore,
\begin{eqnarray}
\nonumber F^{-1}_\omega(S(\alpha_\omega^-(-\tau)))&=&S (S(F^{-1}_\omega(S(\alpha_\omega^-(-\tau)))))= S( F_\omega(\alpha_\omega^-(-\tau)))\\
\nonumber &=& S(\alpha_\omega^-(-\tau+1))=S(\alpha_\omega^-(-(\tau-1))),
\end{eqnarray}
and the curve $\alpha_\omega^+(\tau):=S(\alpha_\omega^-(-\tau))$ is a stable separatrix of $\Fe$. The conclusion follows.
\end{proof}

\begin{cor}
The functions $\alpha_\omega^\pm$  have analytic extensions to $\CC \setminus \RR_\mp$.
\end{cor}
\begin{proof}
Use the equations  $(\ref{eq:xit})-(\ref{eq:etat})$  to extend  $\alpha_\omega^-$ analytically to $\CC \setminus \RR_+$. Use the reversibility's of the maps $F_0$ and $F_\eps$ to extend  $\alpha_\omega^+$  analytically to $\CC \setminus \RR_-$.
\end{proof}

A straightforward computation demonstrates that 
$$\alpha_\omega^+(\tau)= \left( {a_1^+ \over \tau} + {a_2^++\xi_\omega^+(\tau) \over \tau^2}, {b_2^+ \over \tau^2} + {b_3^++\eta_\omega^+(\tau) \over \tau^3}  \right),$$
where
\begin{equation}
\label{eq:coeffabp} a_1^+=- {i \over 2}, \quad a_2^+=a_2^-, \quad b_2^+={i \over 2}+{1 \over 2}, \quad b_3^+=(2 i-2) a_2^++{1 \over 2},
\end{equation}
and $a_2^+$ is a free parameter.

We recall that the map  
$$P(u,v)=T \circ R \circ T^{-1}(u,v)=(-u+{v \over 2}, v)$$
is a reversor for the map  $G_\eps=T \circ H_\eps \circ T^{-1}$. Therefore, the pair $\beta^\mp_\omega(\tau):=P(\alpha^\pm_\omega(-\tau))$ is also a pair of stable/unstable separatrices of $\Fe$. In particular, 
$$\beta^+_\omega(\tau)=P(\alpha_\omega^-(-\tau))=\left(-{i \over 2 \tau} + {s \over \tau^2} + \ldots, {i+1 \over 2 \tau^2   }+ \left((2 i -2   ) s    +{1 \over  2} \right) {1 \over \tau^3} + \ldots   \right).$$ 
Notice, the coefficients above correspond to the lower choice of signs in formulas $(\ref{eq:coeffab})$. 

We will summarize the results of this Section in one theorem.
\begin{exthm}
Let $\omega \in \{0,\eps\}$, and let $\delta>0$  be fixed. Then, for every $\rho>0$ and $0<\vareps<1$ there exists $A>0$ and a functions $\alpha_\omega^\pm$ and $\beta_\omega^\pm$ which solve the equations $(\ref{eq:xit})-(\ref{eq:etat})$  in $\cB_\rho(0) \subset \cX_{2-\vareps}(\cD_{A,\delta}^\pm) \times \cX_{3-\vareps}(\cD_{A,\delta}^\pm)$. 

Additionally, these functions extend analytically to $\CC \setminus \RR_\mp$, and satisfy
$$
\alpha_\omega^\pm(\tau)=S(\alpha_\omega^\mp(-\tau)), \quad \beta_\omega^\pm(\tau)=P(\alpha_\omega^\mp(-\tau)),$$
and
$$G_\eps(\alpha_\omega^\pm) \subset \beta_\omega^\pm, \quad G_\eps(\beta_\omega^\pm) \subset \alpha_\omega^\pm.$$
\end{exthm}

\section{Exponential bound on the difference of separatrices} \label{sec:difference} 

We consider the difference between the two separatrices $w=(u,v)=\alpha_0^+-\alpha_0^-=(x_0^+ -x_0^-,y_0^+-y_0^-)$ (cf. (\ref{alphaomega})) on the common domain of definition. It is straightforward to demonstrate that 
it satisfies the following equations.
\begin{eqnarray}
\nonumber u(\tau+1)&=& u(\tau) \left(1+ 4 x_0^-(\tau)- 2 y_0^-(\tau) -2 v(\tau) +2 u(\tau)+ p_u(\tau)  \right) +\\
\nonumber  &+& v(\tau) \left(1 -2 x_0^-(\tau) -y_0^-(\tau)  -{1 \over 2} v(\tau) +p_v(\tau) \right) \\
\nonumber  &=& u(\tau) \left(1+ r_u(x_0^-(\tau),y_0^-(\tau),u(\tau),v(\tau))  \right) + \\
\label{eq:ueq}  &+&v(\tau) \left(1 +  s_u(x_0^-(\tau),y_0^-(\tau),u(\tau),v(\tau))  \right), \\
\nonumber  v(\tau+1)&=& u(\tau) \left(-4 y_0^-(\tau) -4 v(\tau) +q_u(\tau)  \right) + \\
\nonumber &+& v(\tau) \left(1 -4 x_0^-(\tau) -2 y_0^-(\tau) -v(\tau)+ q_v(\tau) \right) \\
\nonumber  &=& u(\tau) \ r_v(x_0^-(\tau),y_0^-(\tau),u(\tau),v(\tau)) \\
 \label{eq:veq}  &+& v(\tau) \left(1 +  s_v(x_0^-(\tau),y_0^-(\tau),u(\tau),v(\tau))  \right),
\end{eqnarray}
where $p_{u/v}$ and $q_{u/v}$ are some finite order polynomials that contain only the 
monomials $(x_0^-)^i (y_0^-)^j u^k v^n$,  $i+j+k+n \ge 2$,  and $r_{u/v}$, $s_{u/v}$  are some finite order polynomials 
without constant terms. 

Let $A$, $\delta>0$, $\vareps>0$ be as in Propositions \ref{prop:unst_existence} and \ref{prop:st_existence}.  Notice that for $( u, v)=\alpha_0^+-\alpha_0^- $, where $\alpha_0^\pm$ are as in Theorem~B, 
the function $r_v$ is in $\cX_{2-\varepsilon}(D_{A,\delta})$, where 
$$D_{A,\delta}= \left(D^+_{A,\delta} \cap D^-_{A,\delta}\right) \cup \{\tau \in \CC: \Im(\tau)<0  \}.$$
We rewrite the equations (\ref{eq:ueq}-\ref{eq:veq}) as
\begin{eqnarray}
\nonumber u(\tau+1)&=&  u(\tau) \left(1+ m(\tau)  \right) + v(\tau) \left(1 +  b(\tau) \right), \\
\nonumber  v(\tau+1)&=& u(\tau)  \ c(\tau) + v(\tau) \left(1 +  d(\tau)  \right),
\end{eqnarray}
where $m, b, d \in \cX_{2-\varepsilon}(D_{A,\delta})$ and $c\in \cX_{3-\varepsilon}(D_{A,\delta})$.
After the change $\hat u(\tau) = u(\tau), \; \hat v(\tau) = a_{21} (\tau) u(\tau) +  v(\tau)$
with
$$
a_{21}=\frac{m(\tau)-d(\tau) + \sqrt{\left(m(\tau)-d(\tau)\right)^2 + 4 \ c(\tau) \left(1+b(\tau)\right)}}{2 \left(1+b(\tau)\right)},
$$
we transform 
the above system into an upper-diagonal form 
\begin{eqnarray}
\label{eq:hatu} \hat u(\tau+1) &=& \hat u(\tau) \left(1+ \hat f_u(\tau) 
\right) + \hat v(\tau) \left(1 +  \hat g_u(\tau) 
\right), \\
\label{eq:hatv} \hat v(\tau+1) &=&  \hat v(\tau) \left(1 +  \hat g_v(\tau) 
\right),
\end{eqnarray}
where  the functions
$\hat f_u, \hat g_u, \hat g_v$ are in $\cX_{2-\varepsilon}(D_{A,\delta}). $ 
We, therefore, consider the  solutions of the equations~(\ref{eq:hatu})-(\ref{eq:hatv})
and eliminate the variable $v$ from the first equation. We substitute the second equation  (\ref{eq:hatv})  into the first  (\ref{eq:hatu}) to obtain
\begin{equation}
\label{eq:hatunew} \hat u(\tau+1) = \hat u(\tau) \left(1+ \hat f_u(\tau)  \right) + \hat v(\tau-1) \left(1 +  \hat g_u(\tau)  \right) \left(1 + \hat g_v(\tau-1)  \right).
\end{equation}
Next, we express $\hat v(\tau-1)$ from the first equation (\ref{eq:hatu}),
\begin{equation}
\label{eq:vhat}\hat v(\tau-1)={\hat u(\tau) - \hat u(\tau-1) (1 +\hat f_u(\tau-1) )  \over 1+\hat g_u(\tau-1)  }
\end{equation}
and substitute into  (\ref{eq:hatunew}), to obtain
\begin{eqnarray}
\nonumber \hat u(\tau\hspace{-3mm}&\hspace{-3mm}+\hspace{-3mm}&\hspace{-3mm}1) \left(1+\hat g_u(\tau-1)\right) + \hat u(\tau-1) \left((1+\hat f_u(\tau-1) \right) \left(1+ \hat g_v(\tau-1)\right) \left(1+ \hat g_u(\tau)\right) \\
\nonumber &-& \hat u(\tau) \left((1+\hat f_u(\tau)) (1+\hat g_u(\tau-1))+(1+\hat g_v(\tau-1))(1+\hat g_u(\tau))\right)=0 \implies \\
\label{eq:hatuso}\hat u(\tau\hspace{-3mm}&\hspace{-3mm}+\hspace{-3mm}&\hspace{-3mm}1)  \left(1+f_{+1}(\tau)\right) +\hat u(\tau-1) \left((1+f_{-1}(\tau) \right)+ \hat u(\tau) \left(-2+f_0(\tau)\right) =0, 
\end{eqnarray}
where $f_{\pm 1}$ and $f_0$ are some functions in $\cX_\gamma(D_{A,\delta})$, $\gamma=\min \{\mu, 2-\vareps \}$.  
As such, $f_{\pm 1}$ and $f_0$ have no constant terms, 
$$f_i(\tau)={c^i_1 \over \tau}+O\left({1 \over \tau^2} \right), \ i=-1,0,1,$$
with
$$c^{-1}_1=-3 i, \quad c_1^0=4 i, \quad c^1_1=-i.$$
The equation (\ref{eq:hatuso}) has two functionally independent solutions, which have the following representation 
(cf. \cite{Hun})
\begin{eqnarray}
\label{eq:phipm} \phi_\pm(\tau)&=&e^{\pm a \tau^{1 \over 2}} \tau^r \left(1+\sum_{k=1}^\infty {b_k^\pm  \over \tau^{k \over 2} }   \right), \\
\nonumber a&=&2 \sqrt{2 c^0_1-c^{-1}_1-c^{1}_1 }=4 \sqrt{3 i},\\ 
\nonumber r&=& {c^{-1}_1 -c^1_1\over 2}+{1 \over 4}=-i +{1 \over 4}.
\end{eqnarray}
Substitution of the above into the equation (\ref{eq:hatuso}) gives,
$$ b_1^\pm \hspace{-1mm}= \hspace{-1mm} \mp {1 \over 48} { 32 (c^{-1}_1)^2  \hspace{-0.6mm} - \hspace{-0.6mm}  32 (c^{-1}_1 \hspace{-0.6mm}  -\hspace{-0.6mm}   3) c^{0}_1 \hspace{-0.6mm}  - \hspace{-0.6mm}  16 (c^{0}_1)^2 \hspace{-0.6mm}  + \hspace{-0.6mm}   8 (2 c^{-1}_1 \hspace{-0.6mm}  - \hspace{-0.6mm}  4 c^{0}_1 \hspace{-0.6mm}  -\hspace{-0.6mm}  9) c^{1}_1 \hspace{-0.6mm}  + \hspace{-0.6mm}  32 (c^{1}_1)^2 \hspace{-0.6mm}  - \hspace{-0.6mm}  24 c^{-1}_1 \hspace{-0.6mm}  + \hspace{-0.6mm}  9 \over \sqrt{-c^{-1}_1 + 2 c^{0}_1 - c^{1}_1} }.$$

Our goal in this section will be to show that $u$ and $v$, or, equivalently, $\hat u$ and $\hat v$ vanish exponentially 
as $|\Im(\tau)|$ grows. To demonstrate this, we would first require several results about solutions of the 
second-order homogeneous difference equations.

Consider the equation ~(\ref{eq:hatuso}),  
where all functions are  defined in some domain $D \subset \CC$. A solution of  the equation (\ref{eq:hatuso}) 
is a function defined in $D$ and satisfying the equation at any point $\tau$, such that $\tau$, $\tau-1$, $\tau+1$  
belong to $D$. We will call a function $\alpha$ defined in $D$, such that $\alpha(\tau+1)=\alpha(\tau)$ for all $\tau$ 
such that  $\tau+1$ is also in $D$. 

Following \cite{La1}, define the operators
\begin{eqnarray}
\nonumber (\Delta f)(\tau)&=&f(\tau+1)-f(\tau), \\
\nonumber (\bar \Delta f)(\tau)&=&f(\tau)-f(\tau-1), \\
\nonumber (\Delta^2 f)(\tau)&=&f(\tau+1)+f(\tau-1)-2 f(\tau),\\
\nonumber f`(\tau)&=&f(\tau+1), \\
\nonumber `f(\tau)&=&f(\tau-1).
\end{eqnarray}

It is straightforward to demonstrate that
\begin{eqnarray}
\nonumber \Delta^2 &=&\Delta \bar \Delta =\bar  \Delta  \Delta,\\
\nonumber \bar \Delta f `&=& (\bar \Delta f) `= \Delta f,\\
\nonumber \Delta `f&=& `(\Delta f)= \bar \Delta f, \\
\nonumber \Delta f \cdot g &=& \Delta f \cdot g` + f \cdot \Delta g, \\
\nonumber \bar \Delta f \cdot g &=& \Delta f \cdot `g +f \cdot  \bar \Delta g.
\end{eqnarray}

The Wronskian of two functions $f$ and $g$ will be defined as
$$W(f,g)(\tau)=f(\tau) (g(\tau+1)-g(\tau))-g(\tau) (f(\tau+1)-f(\tau))=f(\tau) (\Delta g)(\tau) - g(\tau) (\Delta f)(\tau).$$

In this notation, the equation  (\ref{eq:hatuso}) assumes the form
\begin{equation}
\label{eq:secorder3} (1+f_{-1}) (\Delta^2 \hat u) (\tau) +(f_{+1}-f_{-1}) (\Delta \hat u)(\tau)+ (f_0+f_{+1} + f_{-1})\hat u(\tau) =0.
\end{equation}
Assume that $A$ is sufficiently large, so that the function $(1+f_{-1})^{-1}$ is analytic on $D_{A,\delta}$. Then, the equation  (\ref{eq:secorder3}) is of the form
\begin{equation}
\label{eq:secorder2} (\Delta^2 \hat u) (\tau) +w(\tau) (\Delta \hat u)(\tau)+ z(\tau) \hat u(\tau) =0,
\end{equation}
where $w$ and $z$ are in $\cX_\gamma(D_{A,\delta})$.
\begin{lem}\label{Lazlemma}
If $\phi_\pm$ are two functionally independent solutions of the equation (\ref{eq:secorder2}), such that $W(\phi_+,\phi_-)$ is non-zero everywhere in $D$,
then the general solution is given by
$$\phi(\tau)= \alpha_-(\tau) \phi_{-}(\tau)+\alpha_+(\tau) \phi_+(\tau),$$
where 
\begin{equation}
\label{eq:alpha}\alpha_\pm=\mp {W(\phi,\phi_\mp) \over W(\phi_+,\phi_-)}
\end{equation}
are periodic functions.
\end{lem}
\begin{proof}
Let $\phi$, $\phi_+$ and $\phi_-$ be any three solutions of (\ref{eq:secorder2}), such that $W(\phi_+,\phi_-) \ne 0$. Then, a straightforward computations demonstrates that 
$$\phi={ W(\phi,\phi_+) \over W(\phi_+,\phi_-)}  \phi_- - { W(\phi,\phi_-) \over W(\phi_+,\phi_-)}  \phi_+.$$
We consider the function $\alpha_-=W(\phi,\phi_+) / W(\phi_+,\phi_-)$.
\begin{eqnarray}
\nonumber \bar \Delta \alpha_-&=&\bar \Delta { W(\phi,\phi_+) \over W(\phi_+,\phi_-)} \\
\nonumber &=&{ W(\phi,\phi_+) \over W(\phi_+,\phi_-)}-{ `W(\phi,\phi_+) \over `W(\phi_+,\phi_-)} \\
\nonumber &=&{ W(\phi,\phi_+) `W(\phi_+,\phi_-) - `W(\phi,\phi_+) W(\phi_+,\phi_-) \over W(\phi_+,\phi_-) `W(\phi_+,\phi_-) } \\
\label{eq:barDW} &=&{  W(\phi_+,\phi_-) \bar \Delta W(\phi,\phi_+) -  W(\phi,\phi_+) \bar \Delta W(\phi_+,\phi_-)   \over W(\phi_+,\phi_-) `W(\phi_+,\phi_-) }.
\end{eqnarray}
We will therefore require an expression for $\bar \Delta W(\phi,\psi)$ where $\phi$ and $\psi$ are any two solutions of  (\ref{eq:secorder2}).
\begin{eqnarray}
\nonumber \bar \Delta W(\phi,\psi)&=&\bar \Delta \left(\phi \Delta \psi - \psi \Delta \phi  \right) \\
 \nonumber &=&\bar \Delta \phi \bar \Delta \psi+ \phi \Delta^2 \psi  -  \bar \Delta \psi \bar \Delta \phi -\psi \Delta^2 \phi \\
 \nonumber &=&\phi \left(- w \Delta \psi - z \psi  \right)  - \psi \left(-w \Delta \phi -z \phi  \right) \\
 \nonumber &=& - w W(\phi,\psi).
\end{eqnarray}
Therefore, (\ref{eq:barDW}) becomes 
\begin{equation}
\nonumber \bar \Delta \alpha_-={ -w  W(\phi_+,\phi_-)  W(\phi,\phi_+) +w   W(\phi,\phi_+)  W(\phi_+,\phi_-)   \over W(\phi_+,\phi_-) `W(\phi_+,\phi_-) }=0,
\end{equation}
and $\alpha_-$ is a periodic function. Similarly for $\alpha_+$.
\end{proof}

Now we are ready to demonstrate the exponential bound on a solution of the equation (\ref{eq:secorder2}).
\begin{prop} \label{expprop}
Let $\hat u \in \cX_{\mu}(D_{A,\delta})$ be a solution of the equation (\ref{eq:secorder2}), then for every $\kappa>0$  there exist constants $C$ and  $K$, such that
\begin{eqnarray}
\label{expboundu1} |\hat u(\tau) | &\le& C |\tau|^{1 \over 4}  \| \hat u \|_{\mu} e^{-(2 \pi-\kappa) |\Im(\tau) |}, \\
\label{expboundu2} |\hat u(\tau) | &\ge& K |\tau|^{1 \over 4} \min_{|\Re(x)| \le {1 \over 2}, \atop \Im (x) = -(A+1) \tan \delta}  \left\{  |\hat u(\tau)| \right\}   e^{-(2 \pi+\kappa) |\Im(\tau) |}
\end{eqnarray}
hold for all $\tau \in D_{A,\delta}$
\end{prop}
\begin{proof}
Consider the Wronskian $W(\phi_+,\phi_-)$, where $\phi_\pm$ are as in (\ref{eq:phipm}).
\begin{eqnarray}
\nonumber W(\phi_+,\phi_-)&=&e^{a \tau^{1 \over 2}} \tau^r \left(1+\sum_{k=1}^\infty {b_k^+  \over \tau^{k \over 2} }   \right) \Delta e^{- a \tau^{1 \over 2}} \tau^r \left(1+\sum_{k=1}^\infty {b_k^-  \over \tau^{k \over 2} }   \right) \\
\nonumber &-& e^{-a \tau^{1 \over 2}} \tau^r \left(1+\sum_{k=1}^\infty {b_k^-  \over \tau^{k \over 2} }   \right) \Delta e^{ a \tau^{1 \over 2}} \tau^r \left(1+\sum_{k=1}^\infty {b_k^+  \over \tau^{k \over 2} }   \right) \\
\nonumber &=&\left( e^{a \tau^{1 \over 2}} \tau^r  \left( e^{-a (\tau+1)^{1 \over 2}} (\tau+1)^r -e^{-a \tau^{1 \over 2}} \tau^r    \right) - \right.\\
\nonumber &\phantom{=}& \left.- e^{-a \tau^{1 \over 2}} \tau^r  \left( e^{a (\tau+1)^{1 \over 2}} (\tau+1)^r -e^{a \tau^{1 \over 2}} \tau^r    \right) \right) \left( 1 + O \left( \tau^{-{k \over 2}} \right) \right)\\
\nonumber &=&(\tau+1)^r \tau^r \left( e^{a \tau^{1 \over 2}} e^{-a (\tau+1)^{1 \over 2}} - e^{-a \tau^{1 \over 2}} e^{a (\tau+1)^{1 \over 2}} \right) \left( 1 + O \left( \tau^{-{k \over 2}} \right) \right) \\
\nonumber &=&(\tau+1)^r \tau^r \left( e^{-a \left({1 \over 2 \tau}   + O\left({1 \over \tau^2} \right) \right)} -  e^{a \left({1 \over 2 \tau}   + O\left({1 \over \tau^2} \right)\right)} \right) \left( 1 + O \left( \tau^{-{k \over 2}} \right) \right)\\
\label{eq:Wron} &=&- (\tau+1)^r \tau^{r-{1\over 2}} a \left(1+ O\left({1 \over \tau} \right) \right)  \left( 1 + O \left( \tau^{-{k \over 2}} \right) \right).
\end{eqnarray}
We see that the above is non-zero on $D_{A,\delta}$ if $A$ is sufficiently large.  

If $\hat u$ is any solution of the equation (\ref{eq:secorder2}), then according to Lemma \ref{Lazlemma}, 
$$\hat u=\alpha_-  \phi_- + \alpha_+   \phi_+, \ \alpha_\pm = \mp { W(\hat u ,\phi_\mp) \over W(\phi_+,\phi_-)}.$$
Therefore, using  (\ref{eq:phipm}), (\ref{eq:alpha}) and (\ref{eq:Wron})
\begin{equation}
\label{eq:alphabound1} |\alpha_\pm(\tau)|=\left| {W(\hat u,\phi_\mp) \over W(\phi_+,\phi_-)} \right|=  O\left( { |\hat u(\tau)| \left| e^{\mp a \tau^{1 \over 2}} \tau^r \right|  \over | \tau^{2 r-{1\over 2}} | } \right)= O\left( { |\hat u(\tau)| e^{\mp \Re( a \tau^{1 \over 2})} \over \left| \tau^{r-{1\over 2}} \right| } \right).
\end{equation}

The function $\hat u$ is in $\cX_{2-\vareps}(D_{A,\delta})$, therefore, the norm $\| \hat u\|_{2-\vareps}$ is bounded, and 
\begin{equation}
\label{eq:alphabound} |\alpha_\pm(\tau)|= O\left( { \|\hat u(\tau)\|_{2-\vareps} e^{\mp \Re( a \tau^{1 \over 2})} \over \left| \tau^{r+2-\vareps - {1\over 2}} \right| } \right).
\end{equation}

Since $\alpha_\pm$ are periodic, the following functions are well-defined
$$\beta_\pm(z)=\alpha_\pm \left({i \over 2 \pi} \ln (z)  \right)$$ 
on $\DD_R \setminus \{ 0\}$ where $R=e^{-2 \pi (A+1) \tan \delta}$. Notice, our chosen branch of the logarithm maps the set $\DD_R \setminus \{ 0\}$ onto the semi-infinite strip $S_{A,\delta}=\{z \in \CC: 1/2 < \Re{(z)} \le 1/2, \Im{(z)}<-(A+1) \tan \delta  \}$.  By (\ref{eq:alphabound}), $\alpha_\pm \rightarrow 0$ in $S_{A,\delta}$ as $\Im(\tau) \rightarrow -\infty$, therefore, $\beta_\pm$ is a well-defined holomorphic function on all of $\DD_R$. We have, denoting, $z=e^{-2 \pi i \tau}$,
\begin{eqnarray}
\nonumber \left|{  \beta_{\pm}(z) \over z } \right| & \le& \max_{|z|=R} \left| {  \beta_{\pm}(z)   \over z } \right| = e^{2 \pi (A+1) \tan \delta} \max_{\tau \in S_{A,\delta}, \atop \Im (\tau) = -(A+1) \tan \delta}  \left|\alpha_{\pm}(\tau) \right| \\
\nonumber &\le&  C_1 \max_{|\Re(\tau)| \le {1 \over 2}, \atop \Im (\tau) = -(A+1) \tan \delta}  \left\{ { |\hat u(\tau)| e^{\mp \Re( a \tau^{1 \over 2})} \over \left| \tau^{r-{1\over 2}} \right| } \right\}  e^{2 \pi (A+1) \tan \delta} \\
\nonumber &\le&  C_1 { \| \hat u  \|_\mu  \over ((A+1) \tan \delta)^{\Re(r)+\mu-{1\over 2}} } \max_{|s| \le {1 \over 2}} \left\{ e^{\mp \Re( a (-i (A+1) \tan \delta +s )^{1 \over 2})} \right\}    e^{2 \pi (A+1) \tan \delta}.
\end{eqnarray}
Similarly,
\begin{eqnarray}
\nonumber \left|{  \beta_{\pm}(z) \over z } \right| & \ge& \min_{|z|=R} \left| {  \beta_{\pm}(z)   \over z } \right| = e^{2 \pi (A+1) \tan \delta} \min_{\tau \in S_{A,\delta}, \atop \Im (\tau) = -(A+1) \tan \delta}  \left|\alpha_{\pm}(\tau) \right| \\
\nonumber &\ge&  K_1 \min_{|\Re(\tau)| \le {1 \over 2}, \atop \Im (\tau) = -(A+1) \tan \delta}  \left\{ { |\hat u(\tau)| e^{\mp \Re( a \tau^{1 \over 2})} \over \left| \tau^{r-{1\over 2}} \right| } \right\}  e^{2 \pi (A+1) \tan \delta} \\
\nonumber &\ge&  K_2 \min_{|\Re(\tau)| \le {1 \over 2}, \atop \Im (\tau) = -(A+1) \tan \delta}  \left\{  |\hat u(\tau)| e^{\mp \Re( a \tau^{1 \over 2})} \right\}  {e^{2 \pi (A+1) \tan \delta} \over ((A+1) \tan \delta)^{\Re(r)-{1\over 2}}}.
\end{eqnarray}
Therefore,
\begin{eqnarray}
\nonumber \left| \alpha_{\pm}(\tau) \right| \hspace{-2mm}&\le&\hspace{-2mm} C_1 { \| \hat u  \|_\mu  \over ((A+1) \tan \delta)^{r+\mu-{1\over 2}} } \max_{|s| \le {1 \over 2}} \hspace{-1mm}  \left\{ e^{\mp \Re( a (-i (A+1) \tan \delta +s )^{1 \over 2})} \right\}  \hspace{-1mm}   e^{2 \pi (A+1) \tan \delta -2 \pi | \Im(\tau) |}, \\
\nonumber \left| \alpha_{\pm}(\tau) \right| \hspace{-2mm}&\ge& \hspace{-2mm}K_2 \hspace{-4mm} \min_{|\Re(\tau)| \le {1 \over 2}, \atop \Im (\tau) = -(A+1) \tan \delta} \hspace{-4mm} \left\{  |\hat u(\tau)| e^{\mp \Re( a \tau^{1 \over 2})} \right\}  {1 \over ((A+1) \tan \delta)^{\Re(r)-{1\over 2}}}  e^{2 \pi (A+1) \tan \delta -2 \pi | \Im(\tau) |}, 
\end{eqnarray}
Finally, we have, using the expressions (\ref{eq:phipm}) for $\phi_\pm$,
\begin{eqnarray}
\nonumber | \alpha_\pm( \tau)  \phi_\pm(\tau)|\hspace{-3mm}& \le & \hspace{-3mm}{ C_2 \| \hat u  \|_\mu  \left| e^{\pm a \tau^{1 \over 2}} \tau^r \right|  \over ((A\hspace{-0.6mm}+\hspace{-0.6mm} 1) \tan \delta)^{r+\mu-{1\over 2}} } \max_{{|s| \le {1 \over 2}}} \hspace{-1mm}\left\{ \hspace{-0.5mm} e^{\mp \Re( a (-i (A\hspace{-0.4mm}+\hspace{-0.4mm}1) \tan \delta \hspace{-0.4mm}+\hspace{-0.4mm}s )^{1 \over 2})} \hspace{-1mm} \right\}    e^{2 \pi (A\hspace{-0.4mm}+\hspace{-0.4mm}1) \tan \delta \hspace{-0.4mm}-\hspace{-0.4mm}2 \pi | \Im(\tau) |}, \\
\nonumber | \alpha_\pm(\tau)  \phi_\pm(\tau)|\hspace{-3mm}& \ge & \hspace{-3mm} K_3 \hspace{-7mm} \min_{|\Re(x)| \le {1 \over 2}, \atop \Im (x) = -(A+1) \tan \delta} \hspace{-7mm} \left\{ \hspace{-1mm}   |\hat u(\tau)| e^{\mp \Re( a \tau^{1 \over 2})} \hspace{-1mm}  \right\} \hspace{-1mm} { \left| e^{\pm a \tau^{1 \over 2}} \tau^r \right|  \over ((A+1) \tan \delta)^{\Re(r)-{1\over 2}}}  e^{2 \pi (A+1) \tan \delta -2 \pi | \Im(\tau) |}, 
\end{eqnarray}
for some constant $C_2$ and $K_3$, depending on $A$ and $\delta$. Notice, that 
$$\left|a \tau^{1 \over 2} \right| < C_3 |\Im(\tau)|^{1 \over 2} +C_4  |\Re(\tau)|^{1 \over 2} \le C_5 |\Im(\tau)|^{1 \over 2},$$
where $C_3$, $C_4$ and $C_5$ are some constants that depend on $a \ne 0$ and $D_{A,\delta}$

Furthermore, $a=4 \sqrt{3 i}$, and elementary geometric considerations demonstrate that $\{ a \tau^{1 \over 2}: \tau \in D_{A,\delta}\} \Subset \{z \in \CC: |\arg(z)|< \nu \} \cup \{z \in \CC: |\arg(z)-\pi|< \nu \}$ for some $\nu < \pi/4$. Therefore, there exists a positive constants $K_4$ such that
$$K_4 |\Im(\tau)|^{1 \over 2}  \ge  |\Re(a \tau^{1 \over 2})|.$$

Therefore,
$$|\hat u(\tau) | \le C_6 |\tau|^{\Re(r)} e^{-\Im(r) \arg(\tau)}  \| \hat u \|_{\mu} e^{2 \pi (A+1) \tan \delta -2 \pi |\Im(\tau) |+C_5 |\Im(\tau)|^{1 \over 2}},$$
and
$$| \hat u(\tau)|\ge  K_3 \min_{|\Re(\tau)| \le {1 \over 2}, \atop \Im (\tau) = -(A+1) \tan \delta}  \left\{  |\hat u(\tau)| \right\}  { \left| 
\tau^r \right|  \over ((A+1) \tan \delta)^{\Re(r)}}  e^{2 \pi (A+1) \tan \delta -2 \pi | \Im(\tau) |-K_4 |\Im(\tau)|^{1 \over 2}}.$$

In particular, for any $\kappa>0$, there exists a constants $C$ and $K$, such that (\ref{expboundu1}) and (\ref{expboundu2})  hold for all $\tau \in D_{A,\delta}$.
\end{proof}

Since the function $\hat v$ is defined through the equation (\ref{eq:vhat}), we  immediately have the following.
\begin{cor}
A solution $\hat v$ of the equation (\ref{eq:hatv}) satisfies a similar bound
\begin{eqnarray}
\label{expboundv1} |\hat v(\tau) | &\le& C |\tau|^{1 \over 4}  \| \hat u \|_{\mu} e^{-(2 \pi-\kappa) |\Im(\tau) |}, 
\end{eqnarray}
possibly with a different constant $C$. 
\end{cor}

Recall, that we are not trying to prove the existence of solutions of equations (\ref{eq:ueq})--(\ref{eq:veq}): the functions $u=x_0^+-x_0^-$ and $v=y_0^+-y_0^-$ are known to be in $\cX_{2-\vareps}(D_{A,\delta})$  and $\cX_{3-\vareps}(D_{A,\delta})$, respectively, and satisfy this system by construction.
The same conclusion holds about the functions $u$ and $v$:
\begin{eqnarray}
\label{eq:expboundun} |u(\tau)| &\le&  K  |\tau|^{1 \over 4} \|x_0^+-x_0^- \|_{2-\vareps} e^{2 \pi ( A-(1-\kappa) 
|\Im(\tau)| ) },\\
\label{eq:expboundvn} |v(\tau)| &\le&   K  |\tau|^{1 \over 4} \|y_0^+-y_0^- \|_{3-\vareps} e^{2 \pi ( A-(1-\kappa) |\Im(\tau)| ) }.
\end{eqnarray}

\section{Exponential splitting} \label{sec:splitting}

We consider the splitting component normal to the unstable separatrix
\begin{equation}
\label{eq:normal}\hat \Theta (\tau) = \det \left[ { d \pi_x \alpha_0^-(\tau) \over d \tau}  \quad   \pi_x(\alpha_0^+(\tau)-\alpha_0^-(\tau))  \atop   { d \pi_y \alpha_0^-(\tau) \over d \tau}  \quad  \pi_y (\alpha_0^+(\tau)-\alpha_0^-(\tau) )   \right].
\end{equation}

In the next Proposition we will require the following result from \cite{La1}. 

Let $\delta \in (0, \pi/2)$ and $B > 4 \tan \delta$. Set 
\begin{equation}\label{eq:DA}
 D^B_\delta=\{z \in \CC: -\pi+\delta \le  \arg z  \le -\delta, \Im z  \le  -B \}.
\end{equation}
Clearly, the set $D^B_\delta$ is a subset of $D_{A,\delta'}$ for some $A$ and $\delta'$.

 \begin{lem} \label{prop:Laprop}
 For any positive $\rho>0$, $\mu>0$, there exists a linear map $\Delta^{-1}: \cX_{\mu+2+\rho}(\cD^B_\delta) \mapsto \cX_\mu(\cD^B_\delta)$ such that
 \begin{itemize}

\item[$i)$] given $g \in \cX_{\mu+2+\rho}(\cD^B_\delta)$, $\Delta^{-1}(g)$ is the solution of the equation 
 $$\Delta(u) :=u(z+1)-u(z)=g(z);$$

\item[$ii)$] the following estimate holds for the norm of $\Delta^{-1}$ 
 \begin{equation} \label{eq:Deltabound}
 \|\Delta^{-1} \| \le C \ B^{-\rho},
 \end{equation}
 where the constant depends only on $\delta$, $\mu$ and $\rho$.
 \end{itemize}
 \end{lem}
\begin{prop}
Let $A$,  $\delta$ and $\vareps$ be as in the Propositions  \ref{prop:unst_existence} and \ref{prop:st_existence}. Let $B$ and $\delta'$ be such that $D^B_{\delta'} \subset D_{A,\delta}$. Then the normal component of splitting is an asymptotically periodic function in $D^B_{\delta'}$:
$$\hat \Theta(\tau) =\Theta_1(F_0) e^{-2 \pi i \tau}+  \hat \Theta_p(\tau) + O\left( e^{-4 \pi |\Im(\tau)|} \right).$$
where for any $\mu>0$ and $\rho>0$,
\begin{equation}\label{theta_norm}
\|\hat \Theta_p\|_\mu  \le   K'  \|\alpha_0^+-\alpha_0^- \|_{2-\vareps}^2 ,
\end{equation}
and
\begin{equation}\label{theta_bound}
|\hat \Theta_p (\tau)| \le K''|\tau|^{{1 \over 2}+\rho+\eps} \|\alpha_0^+-\alpha_0^- \|_{2-\vareps}^2 e^{4 \pi ( A-(1-\kappa) |\Im(\tau)| ) }|,
\end{equation}
where $K'$ and $K''$ depend on  $A$, $B$, $\delta$, $\delta'$, $\mu$, $\rho$, and $\Theta_1$ is a constant that depends on $F_0$.
\end{prop}
\begin{proof}
Denote  $\phi(\tau)= { d \alpha_0^-(\tau)  \over d \tau }$ and $\psi(\tau)=\alpha_0^+(\tau)-\alpha_0^-(\tau)$.

According to equation (\ref{eq:alphaeq}),
\begin{eqnarray}
\nonumber \phi(\tau+1) & \hspace{-2mm}= \hspace{-2mm}& D F_0( \alpha_0^-(\tau) )  \phi(\tau), \\
\nonumber \psi(\tau+1) & \hspace{-2mm}=\hspace{-2mm} & D F_0( \alpha_0^-(\tau) )  \psi(\tau)\hspace{-0.5mm} + \hspace{-0.5mm} \left\{   F_0( \alpha_0^-(\tau) \hspace{-0.5mm} + \hspace{-0.5mm} \psi(\tau) ) \hspace{-0.5mm} - \hspace{-0.5mm}  F_0( \alpha_0^-(\tau) )   \hspace{-0.5mm} - \hspace{-0.5mm} D F_0( \alpha_0^-(\tau) )  \psi(\tau)   \right\}.
\end{eqnarray}

Since $F_0$ is area-preserving, $\det F_0 =1$, and we have, therefore,
\begin{equation}
\label{eq:hatT} \hat \Theta(\tau+1) =  \hat \Theta (\tau) + \det(M(\tau))
\end{equation}
where the matrix-valued function $M(\tau)$ has the form
\begin{eqnarray}
\nonumber M(\tau)&=& \left[ \phi(\tau)  \quad    F_0( \alpha_0^-(\tau) +\psi(\tau) ) -   F_0( \alpha_0^-(\tau) )    -  D F_0( \alpha_0^-(\tau) )  \psi(\tau)  \right] \\
\nonumber &=& \left[{d x_0^-(\tau) \over d \tau}  \quad  O((x_0^+(\tau)-x_0^-(\tau) )^i (y_0^+(\tau)-y_0^-(\tau) )^j ) \atop  {d y_0^-(\tau) \over d \tau}   \quad   O((x_0^+(\tau)-x_0^-(\tau) )^i (y_0^+(\tau)-y_0^-(\tau) )^j ) \right],
\end{eqnarray}
$i+j=2$. Denote $g(\tau):=\det(M(\tau)$. The fact that, by Theorem B $x^-_0 \in \cX_{2-\eps}(D_{A,\delta})$ and $y^-_0 \in \cX_{3-\eps}(D_{A,\delta})$, and that $(D^B_{\delta'} \cup \{\infty\}) \Subset (D_{A,\delta} \cup \{\infty\})$ in $\CC^*$, implies via Cauchy bounds that 
$$ \left|{d x_0^-(\tau) \over d \tau} \right|  \le C_1 {1 \over |\tau|^{2-\eps}}, \quad \left|{d y_0^-(\tau) \over d \tau} \right|  \le C_2 {1 \over |\tau|^{3-\eps}}$$
for some constants $C_1$ and $C_2$.
This, together with the bounds (\ref{eq:expboundun}) and (\ref{eq:expboundvn}), implies that 
\begin{equation}\label{g_exp}
|g(\tau)| \le C {1 \over |\tau|^{{3 \over 2}-\eps}} \|\alpha_0^+-\alpha_0^- \|_{2-\vareps}^2 e^{4 \pi ( A-(1-\kappa) |\Im(\tau)| ) }
\end{equation}
for some constant $C$. The exponential dampening factor in (\ref{g_exp}) means that $g \in \cX_{\mu+2+\rho}(D_{\delta'}^B)$ for any $\mu>0$ and $\rho>0$. Therefore, by Lemma \ref{prop:Laprop}, the equation (\ref{eq:hatT}) has a particular solution $\hat \Theta_p$ which satisfies for any fixed $\mu>0$ and $\rho>0$ 
 $$\|\hat \Theta_p\|_\mu  \le C' (A \tan \delta)^{-\rho}  \|\alpha_0^+-\alpha_0^- \|_{2-\vareps}^2,$$
where $C'=C'(A,B,\delta',\delta,\mu,\rho)$. Furthermore, again, by Lemma \ref{prop:Laprop},
$$|\hat \Theta_p (\tau) \tau^\mu| \le C'' (A \tan \delta)^{-\rho}  {|\tau|^{\mu+2+\rho} \over |\tau|^{{3 \over 2}-\eps}} \|\alpha_0^+-\alpha_0^- \|_{2-\vareps}^2 e^{4 \pi ( A-(1-\kappa) |\Im(\tau)| ) }|,$$
and the bound (\ref{theta_bound}) follows.

The difference $\hat \Theta -\hat \Theta_p$ satisfies  $\Delta \left(\hat \Theta -\hat \Theta_p\right)=0$, and, therefore, is a periodic function in $D^B_{\delta'}$.  We consider the first Fourier coefficient $\Theta_1$  of this periodic function,
$$ \Theta_1=\int_{\tau}^{\tau+1}  e^{2 \pi i s} \left(\hat \Theta(s) -\hat \Theta_p(s)\right) d s, \ \tau, \tau+1 \in D^B_{\delta'}.$$

Finally,
$$\hat \Theta(\tau) =\hat \Theta_p(\tau) +\Theta_1 e^{-2 \pi i \tau}+ O\left( e^{-4 \pi |\Im(\tau)|} \right).$$

The conclusion of the Proposition follows.
\end{proof}

Now, recall that $\tau$ was a reparametrization of time $t$: $\tau=(t-i \pi/2)/h$. This implies that for real $t$
\begin{equation}
\label{eq:mainbound}
|\hat \theta(t)| =|\Theta_1| e^{-{\pi^2 \over h}}+  O\left(h^{-{1 \over 2} -\rho-\eps} \|\alpha_0^+-\alpha_0^- \|_{2-\vareps}^2 e^{-2 (1-\kappa) {\pi^2 \over  h}} \right) + O\left( e^{-2 {\pi^2 \over  h}} \right),
\end{equation}
where $\hat \theta(t):= \hat \Theta \left((t-i \pi/2)/h \right)$.

The Main Theorem follows now from (\ref{eq:mainbound}).

Calculations of this Section and Section \ref{sec:difference} can be repeated verbatim for the differences $\alpha_0^+-\alpha_0^-$ restricted to the domains
$$U_{A,\delta}= \left(D^+_{A,\delta} \cap D^-_{A,\delta}\right) \cup \{\tau \in \CC: \Im(\tau)>0  \}$$
and 
$$U^B_\delta=\{z \in \CC: -\pi+\delta \le  \arg z  \le -\delta, \Im z  \ge  B \}.$$
in the upper half plane, with the same conclusions.

\end{document}